\documentclass{article}

\usepackage[frenchb,english]{babel}

\usepackage{amsmath}
\usepackage{amssymb}
\usepackage{amsthm}
\usepackage{amsfonts} %% Afin d'avoir mathfrak et les lettres gothiques
\usepackage{stmaryrd}
\usepackage{graphicx}
\usepackage{comment}
\usepackage{dsfont}
\usepackage{xcolor}
\usepackage{mathtools}
\usepackage{bigints}
\usepackage{yhmath}
\usepackage{enumitem} %% Pour supprimer le décalage induit automatiquement par l'environnement itemize ou enumerate avec la commande leftmargin=*
\usepackage{subcaption} %% Pour faire des sous-figures dans les figures
\usepackage{appendix} %% Comme son nom l'indique, une fois n'est pas coutume
\usepackage{geometry} %% Pour ajuster les marges
\usepackage{authblk} %% Pour les détails sur l'auteur
\usepackage{empheq} %% Afin de faire des systèmes d'équations parenthésés
\usepackage{tikz} %% Pour faire des dessins, et en particulier pour écrire des nombres dans des cercles

\newtheorem{lemma}{Lemma}

\newtheorem{theor}{Theorem}
\newtheorem{corol}{Corollary}
\newtheorem{defin}{Definition}
\newtheorem{remar}{Remark}

\newcommand{\vertii}[1]{{\left\vert\kern-0.3ex\left\vert #1 
    \right\vert\kern-0.3ex\right\vert}}

\newcommand{\vertiii}[1]{{\left\vert\kern-0.3ex\left\vert\kern-0.3ex\left\vert #1 
    \right\vert\kern-0.3ex\right\vert\kern-0.3ex\right\vert}}
    
\newcommand*\dd{\mathop{}\!\mathrm{d}}

\title{On the local Maxwellians solving the Boltzmann equation with boundary condition}
\author{Th\'eophile Dolmaire}
\affil{Dipartimento di Ingegneria e Scienze dell’Informazione e Matematica (DISIM),\\Università degli Studi dell’Aquila,\\Edificio Renato Ricamo, via Vetoio, Coppito, 67100 L’Aquila, Italy}
\begin{document}

\maketitle

\begin{abstract}
\noindent
We derive the expressions of the local Maxwellians that solve the Boltzmann equation in the interior of an open domain. We determine which of these local Maxwellians satisfy the Boltzmann equation in a regular domain with boundary, without assuming the boundedness of the domain. We investigate separately, on the one hand, the case of the bounce-back boundary condition in any dimension, and on the other hand the case of the specular reflection boundary condition, in dimension $d = 2$ and $d = 3$. In the case of the bounce-back boundary condition, we prove that the only local Maxwellians solving the Boltzmann equation with boundary condition are the global Maxwellians. In the case of the specular reflection, we provide a complete classification of the domains for which only the global Maxwellians solve the Boltzmann equation with boundary condition, and we describe all the local Maxwellians that solve the equation for the domains presenting symmetries.
\end{abstract}


\section{Introduction}
\numberwithin{equation}{section}

\noindent
In 1872, Ludwig Boltzmann \cite{Bolt964} introduced the following equation, in order to describe dilute gases composed with particles that interact with each other via binary collisions:
\begin{align}
\label{EQUAT_Boltzmann_}
\partial_t f + v\cdot\nabla_x f = \int_{v_*\in\mathbb{R}^d}\int_{\omega\in\mathbb{S}^{d-1}} B\big(\vert v-v_* \vert,\left\vert \frac{v-v_*}{\vert v-v_* \vert} \cdot \omega \right\vert \big) \left[ f(v')f(v'_*) - f(v)f(v_*)\right] \dd \omega \dd v_*.
\end{align}
Here, the unknown $f=f(t,x,v)$ of the Boltzmann equation \eqref{EQUAT_Boltzmann_} is a probability density for all time $t$ on the space $\Omega \times \mathbb{R}^d \ni (x,v)$ (the particles constituting the gas are assumed to evolve in the domain $\Omega \subset \mathbb{R}^d$, where $d$ is an integer, representing the dimension of the physical space that is considered), and:
\begin{align}
f(t,x,v) \dd x \dd v
\end{align}
represents the number of particles of the gas lying in the infinitesimal volume $x + \dd x$ and moving with a velocity in $v + \dd v$ at time $t$. $v'$ and $v_*'$, the \emph{post-collisional velocities}, are defined as:
\begin{align}
\label{EQUATIntroLoiDeCollision_}
v' = v - (v-v_*)\cdot\omega\omega, \hspace{10mm} v_*' = v_* + (v-v_*)\cdot\omega\omega,
\end{align}
$\omega$ is the \emph{angular parameter}, and $B\big(\vert v-v_* \vert,\left\vert \frac{v-v_*}{\vert v-v_* \vert} \cdot \omega \right\vert\big)$, the \emph{collision kernel}, describes the rate at which collisions involving two particles with relative velocities $v-v_*$ and colliding with an angular parameter $\omega$ take place.\\
For instance, if we assume that the particles interact as hard spheres, the angular parameter $\omega$ represents the direction of the line joining the respective centers of the two colliding particles, and in such a case the collision kernel $B$ is:
\begin{align}
B(\vert v-v_* \vert,\left\vert \frac{v-v_*}{\vert v-v_* \vert} \cdot \omega \right\vert) = \vert (v-v_*)\cdot\omega \vert.
\end{align}
To complete the description of the model, boundary conditions need to be added to \eqref{EQUAT_Boltzmann_}. We will assume that the domain $\Omega$ is an open set, with a boundary $\partial\Omega$ regular enough such that at any point $x\in\partial\Omega$ we can define an outgoing unitary normal vector $n(x)$ to the boundary. In the present article, we will consider the two following boundary conditions: the bounce-back boundary condition, defined as
\begin{align}
\label{EQUATIntroBounceBack_B_C_}
\forall t \in [0,T],\, x \in \partial \Omega,\, v \in \mathbb{R}^d,\hspace{3mm} f(t,x,v) = f(t,x,-v).
\end{align}
and the specular reflection boundary condition, defined as
\begin{align}
\label{EQUATIntroSpecuRefle_B_C_}
\forall t \in [0,T], x \in \partial \Omega, v \in \mathbb{R}^d, f(t,x,v) = f(t,x,v'),
\end{align}
with
\begin{align}
\label{EQUATIntroSpecuRefle_v'__}
v' = v - 2 \left( v \cdot n(x) \right) n(x).
\end{align}
These two conditions, written at the mesoscopic level (that is, for the unknown $f$ of the Boltzmann equation \eqref{EQUAT_Boltzmann_}), reflect the model chosen to describe the dynamics of the particles that compose the gas, when such a particle collides with the boundary $\partial\Omega$. In the case of the bounce-back boundary condition, the velocity $v$ of a particle about to leave the domain $\Omega$ is immediately changed into $v'=-v$, while in the specular reflection case, the reflected velocity $v'$ is defined by \eqref{EQUATIntroSpecuRefle_v'__}. For other possible boundary conditions and discussions concerning the relevance of the laws \eqref{EQUATIntroBounceBack_B_C_} and \eqref{EQUATIntroSpecuRefle_B_C_}, the reader may refer to \cite{Cerc988}, \cite{CeIP994} and \cite{Vill002}.\\
\newline
The collision law \eqref{EQUATIntroLoiDeCollision_} describes elastic collisions: momentum and kinetic energy are conserved when two particles collide. These conserved quantities can be recovered at the level of the Boltzmann equation: if $f$ is a solution of \eqref{EQUAT_Boltzmann_} (at least, in the absence of boundary), then
\begin{align}
\frac{\dd}{\dd t} \int_x\int_v \begin{pmatrix} 1 \\ v \\ \vert v \vert^2 \end{pmatrix} f(t,x,v) \dd v \dd x = 0.
\end{align}
More surprisingly, if we define the \emph{entropy} $H$ as the functional:
\begin{align}
H(f)(t) = \int_x\int_v f(t,x,v) \ln f(t,x,v) \dd v \dd x,
\end{align}
then if $f$ is a solution of the Boltzmann equation, we have
\begin{align}
\frac{\dd}{\dd t} H(f)(t) = - \int_x D(f)(x) \dd x \leq 0,
\end{align}
with $D$ the \emph{entropy production} defined as
\begin{align}
D(f) = \frac{1}{4} \int_x\int_v\int_{v_*}\int_\omega B \left[ f(v')f(v'_*)-f(v)f(v_*)\right] \ln \frac{f(v')f(v'_*)}{f(v)f(v_*)}.
\end{align}
In addition, we have that $\int_x D(f)(x) \dd x = 0$ if and only if there exist three functions $\rho,a: \Omega \rightarrow \mathbb{R}_+$ and $u: \Omega \rightarrow \mathbb{R}^d$ such that:
\begin{align}
f(x,v) = \rho(x) \exp\left( -a(x) \vert v - u(x) \vert^2 \right).
\end{align}
These two statements constitute the celebrated \emph{H-theorem}, discovered by Boltzmann (\cite{Bolt964}). In addition, let us observe that any function of the form:
\begin{align}
\label{EQUATIntroMaxweGloba}
m(t,x,v) = \rho_0 \exp\left(-a_0\vert v-u_0 \vert^2\right),
\end{align}
where $\rho_0,a_0 \in \mathbb{R}_+$, $u_0 \in \mathbb{R}^d$ are constant, is a solution to the Boltzmann equation. A function of the form \eqref{EQUATIntroMaxweGloba} is called a \emph{global Maxwellian}, whereas a function of the form:
\begin{align}
\label{EQUATIntroMaxweLocal}
m(t,x,v) = \rho(t,x) \exp\left( -a(t,x) \vert v - u(t,x) \vert^2 \right),
\end{align}
with $\rho,a: I\times \Omega \rightarrow \mathbb{R}_+$ and $u: I\times\Omega \rightarrow \mathbb{R}^d$ ($I \subset \mathbb{R}$), is called a \emph{local Maxwellian}.\\
\newline
If we can show that the only local Maxwellians solving the Boltzmann equation \eqref{EQUAT_Boltzmann_} are global Maxwellians, then the H-theorem would indicate the long-time behaviour of the solutions of the Boltzmann equation: we would expect that any solution $f$ converges towards the only global Maxwellian that has the same mass, momentum and kinetic energy as $f$. For this reason, determining the local Maxwellians that solve the Boltzmann equation is of central importance in order to understand the long-time behaviour of the solutions of \eqref{EQUAT_Boltzmann_}.\\
\newline
Since a local Maxwellian cancels the collision term of the Boltzmann equation (that is, the right hand side of \eqref{EQUAT_Boltzmann_}), a local Maxwellian $m$ solves the Boltzmann equation if and only if $m$ solves the free transport equation, that is:
\begin{align}
\partial_t m + v\cdot\nabla_x m = 0.
\end{align}
It is well-known how to determine the local Maxwellians that solve the free transport equation, and such a result was already obtained by Boltzmann himself \cite{Bolt012}. The problem is also discussed by Cercignani in \cite{Cerc988}. In this case, an external field acting on the particles is also considered.\\
As a matter of fact, in the case of a domain without boundary, many local Maxwellians, that are not global, solve the free transport equation. More generally, the question of the long-time behaviour of the solutions of the Boltzmann equation is a very challenging problem, that remains an open question nowadays. The difficulties arise from the mixing of the effects of the collision term of the Boltzmann equation on the one hand, and of the transport term $-v\cdot\nabla_x$ on the other hand. Under the action of the collision term, the solutions converge to local Maxwellians (as it can be rigorously proved in the case of the homogeneous Boltzmann equation). However, the collision term is local in $x$, whereas the transport term $-v\cdot\nabla_x f$ acts on the two variables $x$ and $v$ of the phase space, introducing major complications.\\
Nevertheless, let us mention that important results were obtained in \cite{DeVi005} concerning the long-time behaviour of the solutions of the non-homogeneous Boltzmann equation \eqref{EQUAT_Boltzmann_}. For more information on the Boltzmann equation, the reader may refer to the classical references \cite{Cerc988}, \cite{CeIP994} and \cite{Vill002}.\\
\newline
Back to the question of determining the local Maxwellians solving \eqref{EQUAT_Boltzmann_}, the case of a domain with a boundary reduces dramatically the possibilities. A discussion concerning the geometry of the boundary can already be found in the article \cite{Grad965} of Grad. A classification of the boundaries as well as the complete proof of this classification can be found in \cite{Desv990}. Essentially, the result is that if the boundary of the domain does not present any symmetry, then only the global Maxwellians solve the free transport equation in a bounded domain. However, the results of \cite{Desv990} are obtained under the assumption that the domain $\Omega$ is bounded. Besides, to the best of our knowledge, there exists no analogous result for more general domains. Therefore, and considering the central role of the local Maxwellians in the trend to converge to equilibrium for solutions of the Boltzmann equation, it seemed relevant to provide a discussion similar to the one presented in \cite{Desv990} in the case when the domain $\Omega$ is not necessarily bounded.\\
\newline
In the present article, we determine completely the local Maxwellians that solve the Boltzmann equation in a general domain, not necessarily bounded. In particular, when such a domain presents a boundary, we perform the study in the case of the bounce-back boundary condition, and in the case of the specular reflection boundary condition.\\
The plan of the article is the following. In the second section, we derive the expressions of the local Maxwellians that solve the free transport equation in the interior of an open domain. This section follows the references \cite{Cerc988} and \cite{Desv990}. In the third section, we study the local Maxwellians solving \eqref{EQUAT_Boltzmann_} in a domain with a boundary, considering the bounce-back boundary condition. In particular, we prove that only the global Maxwellians are admissible solutions. Finally, in the fourth section we turn to the case of the specular reflection boundary condition. In this section we determine local Maxwellians solving the Boltzmann equation that were not discussed in \cite{Desv990}.

\paragraph{Notations.}

Throughout this article, we will use the following notations.\\
\newline
For $d>1$ any integer larger than $1$, we will denote the $d\times d$ identity matrix by $I_d$:
\begin{align}
I_d = \begin{pmatrix} 1 & 0 & \dots \\ 0 & 1 & \dots \\
\vdots & \vdots & \ddots \end{pmatrix}.
\end{align}
For $m,n,p \geq 1$ three positive integers, the product of the two matrices $M \in \mathcal{M}_{m \times n}$ and $N \in \mathcal{M}_{n\times p}$ will be denoted by:
\begin{align}
MN.
\end{align}
In particular, the product of a $d \times d$ square matrix $M$ with a vector $u \in \mathbb{R}^d$ will be denoted by $Mu$ (the vectors of $\mathbb{R}^d$ are seen as $d\times 1$ matrices).\\
\newline
The transpose of a matrix $M$ will be denoted by $^t \hspace{-0.5mm}M$.\\
\newline
The scalar product of two vectors $u,v \in \mathbb{R}^d$ will be denoted by $u \cdot v$.\\
\newline
Finally, in dimension $d=3$, the cross-product of two vectors $u = (a,b,c)\in\mathbb{R}^3$ and $v = (x,y,z) \in \mathbb{R}^3$ will be denoted by $u \wedge v$, that is:
\begin{align}
u \wedge v = \begin{pmatrix} a \\ b \\ c \end{pmatrix} \wedge \begin{pmatrix} x \\ y \\ z \end{pmatrix} = \begin{pmatrix} bz-cy \\ cx-az \\ ay-bx \end{pmatrix}.
\end{align}
If $A$ is a $3 \times 3$ skew-symmetric matrix of the form:
\begin{align}
A = \begin{pmatrix} 0 & -c & b \\ c & 0 & -a \\ -b & a & 0 \end{pmatrix},
\end{align}
we say that \emph{$A$ is represented by the vector $u$}, or that \emph{the vector $u$ represents the matrix $A$} with $u = (a,b,c)$, that is, $u$ is the only vector of $\mathbb{R}^3$ such that:
\begin{align}
A x = u\wedge x \hspace{5mm} \forall x \in \mathbb{R}^3.
\end{align}

\paragraph{Assumptions on the domain}

In all the rest of the article, the domain will be denoted by $\Omega$, and will be a non empty, connected open set of $\mathbb{R}^d$. When we will say that \emph{$\Omega$ has a boundary}, we will mean that $\partial\Omega$ is non empty, and $\partial\Omega$ is a $\mathcal{C}^1$ manifold of codimension $1$.

\section{Obtaining the expressions of the admissible local Maxwellians}

In this section we will derive the expression of a local Maxwellian, that solves the Boltzmann equation in dimension $d \geq 2$, on the time interval $[0,T]$, in the interior of a domain $\Omega \subset \mathbb{R}^d$, where $\Omega$ is open and regular.\\
The expression of a local Maxwellian is written as:
\begin{align}
\label{EQUATSect2ExpreMaxweLocal}
m(t,x,v) = \rho(t,x) \exp \left( -a(t,x) \left\vert v - u(t,x) \right\vert^2 \right),
\end{align}
where we assume that $\rho,a : \mathbb{R}\times\Omega \rightarrow \mathbb{R}_+$ are strictly positive almost everywhere, and $u:\mathbb{R}\times\Omega \rightarrow \mathbb{R}^d$. We assume also that $\rho$, $a$ and $u$ are differentiable.\\
It is well known that $Q(m,m)=0$ if and only if $m$ is a local Maxwellian (see for instance \cite{Bolt964} or \cite{Vill002}), that is, of the form of \eqref{EQUATSect2ExpreMaxweLocal}. Therefore, a local Maxwellian solves the Boltzmann equation in the domain $\Omega$ if and only if
\begin{align}
\label{EQUATSect2Transport_Libre}
\forall \, (t,x,v) \in [0,T] \times \Omega \times \mathbb{R}^d, \hspace{5mm} \partial_t m + v\cdot\nabla_x m = 0,
\end{align}
that is, the local Maxwellian $m$ solves the free transport equation. It is a classical result (\cite{Bolt012}, \cite{Cerc988}, \cite{Desv990}, \cite{ReVi008}) to determine completely the local Maxwellians that solve the free transport equation. More precisely, such local Maxwellians can be described as follows.

\begin{theor}[Local Maxwellians solving the free transport equation]
\label{THEORSS2.1ExpreMaxweLocGn}
Let $m : [0,T] \times \Omega \times \mathbb{R}^d \rightarrow \mathbb{R}_+$ be a local Maxwellian of the form:
\begin{align*}
m(t,x,v) = \rho(t,x) \exp \left( -a(t,x) \left\vert v - u(t,x) \right\vert^2 \right),
\end{align*}
with $\rho$, $a$ and $u$ differentiable and $u$ twice differentiable in $x$, with $\rho$ and $a$ positive almost everywhere, that solves on $]0,T[\times\Omega\times\mathbb{R}^d$ the free transport equation:
\begin{align*}
\partial_t m + v \cdot\nabla_x m = 0.
\end{align*}
Then, there exist four real number $r_0$, $\alpha$, $\beta$ and $\gamma \in \mathbb{R}$, two vectors $w_1$ and $w_2 \in \mathbb{R}^d$, and a skew-symmetric matrix $\Lambda_0 \in \mathcal{M}_d(\mathbb{R})$ such that:
\begin{align}
\label{EQUATSS2.1ExpreMaxweLocGn}
\forall\, (t,x,v) \in [0,T]\times\Omega\times\mathbb{R}^d, \hspace{2mm} m(t,x,v) &= r_0 \exp \Big( -\alpha \vert x-tv \vert^2 + \beta (x-tv)\cdot v - \gamma \vert v \vert^2 \nonumber\\
&\hspace{20mm} + 2 \left(\Lambda_0(x-tv)\right)\cdot v - 2w_1\cdot(x-tv) +2w_2\cdot v\Big).
\end{align} 
\end{theor}

\begin{remar}
The expression \eqref{EQUATSS2.1ExpreMaxweLocGn} is consistent with the fact that $m$ is a solution of the free transport equation, since it depends only on $v$ and $x-tv$.\\
Let us observe that the expression can be simplified as:
\begin{align}
\label{EQUATSS2.1ExpreMaxweLocGn}
m(t,x,v) &= r_0 \exp \Big( -\alpha \vert x-tv \vert^2 + \beta (x-tv)\cdot v - \gamma \vert v \vert^2 + 2 \left(\Lambda_0 x\right)\cdot v - 2w_1\cdot(x-tv) +2w_2\cdot v\Big),
\end{align}
using that $\Lambda_0$ is a skew-symmetric matrix, so that $(\Lambda_0 v)\cdot v = 0$.
\end{remar}
\noindent
We will prove the result of Theorem \ref{THEORSS2.1ExpreMaxweLocGn}, following the following different steps, dedicating a section for each of them.
\begin{itemize}
\item First, we will derive a system of equations on the coefficients $\rho$, $a$ and $u$ defining the local Maxwellians. This system will enable to deduce directly that $a$ does not depend on the position variable $x$.
\item Second, we will obtain a first expression for the bulk velocity $u$.
\item Third, we will deduce a first expression for the density $\rho$ of the Maxwellian.
\item Fourth, combining the previous expression of $a$, $\rho$ and $u$, we deduce the final expressions of these quantities.
\item Fifth, we conclude and obtain the expression of a local Maxwellian that solves the free transport equation.
\end{itemize}

\subsection{Deriving a system on $\rho$, $a$ and $u$}

We can obtain a system verified by the coefficients $\rho$, $a$ and $u$ defining the local Maxwellian $m$. To do so, we compute directly, for all $(t,x,v) \in\ ]0,T[ \times \Omega \times \mathbb{R}^d$:
\begin{align}
\label{EQUATSS2.1DerivParttMaxwL}
\partial_t m &= \left( \partial_t \rho \right) \exp \left( - a \vert v-u \vert^2 \right) + \rho \partial_t\left( -a \vert v-u \vert^2 \right) \exp \left( - a \vert v-u \vert^2 \right) \nonumber\\
&= \Big[ \partial_t \rho - \rho \left(\partial_t a\right) \vert v-u \vert^2 + 2\rho a \left(\partial_t u\right) \cdot (v-u) \Big] \exp \left( - a \vert v-u \vert^2 \right),
\end{align}
and, in the same way:
\begin{align}
\label{EQUATSS2.1DerivvParxMaxwL}
v \cdot \nabla_x m &= \Big[ v\cdot\nabla_x \rho - \rho v \cdot \left(\nabla_x a \right) \vert v-u \vert^2 + 2\rho a \Big( \sum_{i=1}^d v_i \partial_{x_i} u\cdot (v-u) \Big) \Big].
\end{align}
Using \eqref{EQUATSS2.1DerivParttMaxwL} and \eqref{EQUATSS2.1DerivvParxMaxwL}, \eqref{EQUATSect2Transport_Libre} becomes:
\begin{align}
\label{EQUATSS2.1TrnspLibreCoeff}
-\partial_t \rho + \rho \left(\partial_t a\right) \vert v-u \vert^2 - 2\rho a \left(\partial_t u\right) \cdot (v-u) - v \cdot \nabla_x \rho + \rho v \cdot \left( \nabla_x a \right) \vert v-u \vert^2 - 2\rho a \Big( \sum_{i=1}^d v_i \partial_{x_i} u\cdot (v-u) \Big) = 0.
\end{align}
The equation \eqref{EQUATSS2.1TrnspLibreCoeff} holds for all $t \in\, ]0,T[$, $x \in \Omega$, $v \in \mathbb{R}^d$. Fixing the two first variables $t$ and $x$, \eqref{EQUATSS2.1TrnspLibreCoeff} can be rewritten as a polynomial equation in the $d$ variables
\begin{align}
z_i = v_i - u_i(t,x)
\end{align}
as follows. Identifying the term of degree $3$, we obtained:
\begin{align}
\rho v \cdot \left( \nabla_x a \right) \vert v-u \vert^2 = \rho \left( \sum_{i=1}^d z_i \partial_{x_i} a \right) \left( \sum_{j=1}^d z_j^2 \right) + \rho \left(u \cdot \nabla_x a\right) \left( \sum_{j=1}^d z_j^2 \right),
\end{align}
where a remainder, of degree $2$, appears. Considering now the terms of degree $2$ in \eqref{EQUATSS2.1TrnspLibreCoeff}, we have:
\begin{align}
\rho \left( \partial_t a \right) \vert v-u \vert^2 - 2 \rho a \left( \sum_{i=1}^d v_i \partial_{x_i} u \cdot (v-u) \right) &= \rho \left( \partial_t a \right) \left( \sum_{i=1}^d z_i^2 \right) - 2 \rho a \left( \sum_{i=1}^d \sum_{j=1}^d \partial_{x_i} u_j z_i z_j \right) \nonumber\\
&\hspace{40mm} - 2\rho a \left( \sum_{i=1}^d\sum_{j=1}^d \partial_{x_i}u_j u_i z_j \right),
\end{align}
where, here again, a remainder of degree $1$ appears. Considering finally the two terms of degree $1$ in \eqref{EQUATSS2.1TrnspLibreCoeff} we have:
\begin{align}
- 2 \rho a \left( \partial_t u \right) \cdot (v-u) - v \cdot \nabla_x \rho = - 2 \rho a \left( \sum_{i=1}^d \partial_t u_i z_i \right) - \left( \sum_{i=1}^d \partial_{x_i} \rho z_i \right) - u \cdot \nabla_x \rho,
\end{align}
with $-u \cdot \nabla_x \rho$, a last remainder, of degree $0$, appears. In the end, \eqref{EQUATSS2.1TrnspLibreCoeff} can be rewritten, in terms of the variables $z_i$, as:
\begin{align}
\label{EQUATSS2.1TrnspLbCoePolyn}
0 &= \rho \left( \sum_{i=1}^d z_i \partial_{x_i} a \right) \left( \sum_{j=1}^d z_j^2 \right) \nonumber\\
&\hspace{5mm} + \rho \left(u \cdot \nabla_x a\right) \left( \sum_{i=1}^d z_i^2 \right) + \rho \left( \partial_t a \right) \left( \sum_{i=1}^d z_i^2 \right) - 2 \rho a \left( \sum_{i=1}^d \sum_{j=1}^d \partial_{x_i} u_j z_i z_j \right) \nonumber\\
&\hspace{5mm} - 2\rho a \left( \sum_{i=1}^d\sum_{j=1}^d \partial_{x_i}u_j u_i z_j \right) - 2 \rho a \left( \sum_{i=1}^d \partial_t u_i z_i \right) - \left( \sum_{i=1}^d \partial_{x_i} \rho z_i \right) \nonumber\\
&\hspace{5mm} - u \cdot \nabla_x \rho - \partial_t\rho.
\end{align}
With \eqref{EQUATSS2.1TrnspLbCoePolyn}, we have rewritten \eqref{EQUATSS2.1TrnspLibreCoeff} as a polynomial equation in the variables $z_i$, of (total) degree $3$. Equating all the coefficients to zero, we obtain the following result.

\begin{lemma}[System solved by the coefficients $\rho$, $a$ and $u$]
\label{LEMMESS2.1SystmCoeffrhoau}
Let $m : [0,T] \times \Omega \times \mathbb{R}^d \rightarrow \mathbb{R}_+$ be a local Maxwellian of the form:
\begin{align*}
m(t,x,v) = \rho(t,x) \exp \left( -a(t,x) \left\vert v - u(t,x) \right\vert^2 \right),
\end{align*}
with $\rho$, $a$ and $u$ differentiable, with $\rho$ and $a$ positive almost everywhere, that solves on $]0,T[\times\Omega\times\mathbb{R}^d$ the free transport equation:
\begin{align*}
\partial_t m + v \cdot\nabla_x m = 0.
\end{align*}
Then, the coefficients $\rho$, $a$ and $u$ solve the following system of partial differential equations:
\begin{subequations}
\label{EQUATSS2.1Systeme_rho_a_u}
\begin{empheq}[left=\empheqlbrace]{align}
\rho \left( \partial_{x_i} a \right) &= 0 \hspace{8mm} \forall\, 1 \leq i \leq d, \label{EQUATSS2.1Systemerho_a_u1}\\
\rho \left[ u\cdot \nabla_x a + \partial_t a -2a (\partial_{x_i}u_i) \right] &= 0 \hspace{8mm} \forall\, 1 \leq i \leq d, \label{EQUATSS2.1Systemerho_a_u2}\\
\rho a \left[ \partial_{x_i} u_j + \partial_{x_j} u_i \right] &= 0  \hspace{5mm} \forall\, 1 \leq i < j \leq d,\label{EQUATSS2.1Systemerho_a_u3}\\
2\rho a \nabla_x u_i \cdot u + 2\rho a (\partial_t u_i) + \partial_{x_i}\rho &= 0  \hspace{8mm} \forall\, 1 \leq i \leq d, \label{EQUATSS2.1Systemerho_a_u4}\\
\partial_t \rho + u\cdot \nabla_x \rho &= 0. \label{EQUATSS2.1Systemerho_a_u5}
\end{empheq}
\end{subequations}
\end{lemma}
\noindent
From the first equation of \eqref{EQUATSS2.1Systeme_rho_a_u} and the assumption that $\rho > 0$ almost everywhere, we deduce in particular the following result.

\begin{corol}[First simplification of the coefficient $a$]
\label{COROLSS2.1PremiSimpl__a__}
Under the assumptions of Lemma \ref{LEMMESS2.1SystmCoeffrhoau}, the coefficient $a$ of the local Maxwellian $m$ does not depend on the position variable $x$. In other words, there exists $\sigma_0: [0,T] \rightarrow \mathbb{R}_+$, differentiable, such that
\begin{align}
\label{EQUATSS2.1PremiExpre__a__}
\forall\, (t,x) \in [0,T]\times\Omega, \hspace{5mm} a(t,x) = \sigma_0(t),
\end{align}
which is strictly positive almost everywhere.
\end{corol}

\subsection{Obtaining a first expression for $u$}

Using the simplification concerning $a$ obtained in Corollary \ref{COROLSS2.1PremiSimpl__a__}, we can rewrite the last four equations of the system \eqref{EQUATSS2.1Systeme_rho_a_u} as:
\begin{subequations}
\label{EQUATSS2.1System2_rho_a_u}
\begin{empheq}[left=\empheqlbrace]{align}
\rho \left[ \sigma_0' - 2 \sigma_0 (\partial_{x_i}u_i) \right] &= 0 \hspace{8mm} \forall\, 1 \leq i \leq d, \label{EQUATSS2.1System2rho_a_u1}\\
\rho \sigma_0 \left[ \partial_{x_i} u_j + \partial_{x_j} u_i \right] &= 0  \hspace{5mm} \forall\, 1 \leq i < j \leq d,\label{EQUATSS2.1System2rho_a_u2}\\
2\rho \sigma_0 \nabla_x u_i \cdot u + 2\rho \sigma_0 (\partial_t u_i) + \partial_{x_i}\rho &= 0  \hspace{8mm} \forall\, 1 \leq i \leq d, \label{EQUATSS2.1System2rho_a_u3}\\
\partial_t \rho + u\cdot \nabla_x \rho &= 0. \label{EQUATSS2.1System2rho_a_u4}
\end{empheq}
\end{subequations}
In particular, the two first equations \eqref{EQUATSS2.1System2rho_a_u1} and \eqref{EQUATSS2.1System2rho_a_u2} suggest to consider the matrix $A$, where:
\begin{align}
A = - \frac{\sigma_0'}{2\sigma_0} I_d + \text{Jac}(u),
\end{align}
$I_d$ denoting the $d \times d$ identity matrix, and $\text{Jac}_x(u)$ the Jacobian matrix of the velocity field $u$. By construction, $A$ is skew-symmetric. Therefore, the vector field:
\begin{align}
(t,x) \mapsto W(t,x) = u(t,x) - \frac{\sigma_0'}{2\sigma_0} x
\end{align}
has its Jacobian matrix (in $x$, at $t$ fixed) $\text{Jac}_x W(t,x) = A(t,x)$ which is skew-symmetric, so, provided that $u$ is twice differentiable with respect to the position variable $x$, we deduce that $W$ is an affine mapping in $x$:
\begin{align}
\forall\, (t,x) \in [0,T]\times\Omega,\hspace{5mm} W(t,x) = \Lambda(t) x + C(t),
\end{align}
for some $\Lambda(t) \in \mathcal{M}_d(\mathbb{R})$ and $C(t) \in \mathbb{R}^d$. We deduce a first result concerning the expression of $u$.

\begin{lemma}[First expression of the velocity field $u$]
\label{LEMMESS2.2PremiExpre__u__}
Let $m : [0,T] \times \Omega \times \mathbb{R}^d \rightarrow \mathbb{R}_+$ be a local Maxwellian of the form:
\begin{align*}
m(t,x,v) = \rho(t,x) \exp \left( -a(t,x) \left\vert v - u(t,x) \right\vert^2 \right),
\end{align*}
with $\rho$, $a$ and $u$ differentiable, and $u$ twice differentiable in $x$, with $\rho$ and $a$ positive almost everywhere, that solves on $]0,T[\times\Omega\times\mathbb{R}^d$ the free transport equation:
\begin{align*}
\partial_t m + v \cdot\nabla_x m = 0.
\end{align*}
Then, there exist a non-negative real-valued, differentiable function $t \mapsto \sigma_0(t) \geq 0$, strictly positive almost everywhere, which coincides with the function given in Corollary \ref{COROLSS2.1PremiSimpl__a__}, a time-dependent, skew-symmetric matrix $t \mapsto \Lambda(t) \in \mathcal{M}_d(\mathbb{R})$ and a time-dependent vector $t \mapsto C(t) \in \mathbb{R}^d$ such that the velocity field $u$ can be written as:
\begin{align}
\label{EQUATSS2.2PremiExpre__u__}
u(t,x) = \Lambda(t) x + \frac{\sigma_0'}{2\sigma_0}x + C(t).
\end{align}
\end{lemma}

\begin{remar}
It is a classical and elementary result to establish that if a smooth vector valued function $W:\mathbb{R}^d \rightarrow \mathbb{R}^d$ is such that its Jacobian matrix is everywhere skew-symmetric, then $W$ is an affine mapping. Such a result means that the movement described by the mapping $W$ corresponds to a rigid motion, since $W$ can be seen as an infinitesimal rotation.\\
Let us observe that such a result can be seen as the ``simplest possible'' version of a Korn-type inequality:
\begin{align}
\vert\hspace{-0.25mm}\vert \text{Jac}(W) \vert\hspace{-0.25mm}\vert_{L^p} \leq C_p \vert\hspace{-0.25mm}\vert \text{Jac}^s(W) \vert\hspace{-0.25mm}\vert_{L^p},
\end{align}
where $\text{Jac}^s(W) = \frac{1}{2} \left[ \text{Jac}(W) + \hspace{0.1mm}^t \hspace{-0.5mm}\text{Jac}(W)\right]$ is the symmetric part of the Jacobian matrix of $W$.\\
One of the first versions of such a result can be found in \cite{Korn906}, and was obtained in order to study the deformation of elastic materials. Korn's inequalities are supplemented with appropriate boundary conditions, motivated by physical considerations.\\
It is worthwhile to mention that versions of Korn's inequality have been established in kinetic theory (such as in \cite{DeVi002}, or more recently in \cite{CDHM022}), and used in a crucial manner to study the long-time behaviour of the solutions of the Boltzmann equation \cite{DeVi005}.
\end{remar}

\subsection{Obtaining a second expression for $u$, and a first expression for $\rho$}

Let us now turn to the density $\rho$, in order to deduce another expression for $u$. Since $\rho$ is assumed to be non-zero almost everywhere, we can rewrite \eqref{EQUATSS2.1System2rho_a_u3} as:
\begin{align}
\label{EQUATSS2.3GradiLogar_Rho_}
\frac{\partial_{x_i} \rho}{\rho} &= \partial_{x_i} \left(\ln\rho\right) = -2 \sigma_0 \nabla_x u_i \cdot u - 2 \sigma_0 \left(\partial_t u_i\right).
\end{align}
On the one hand the quantity described in \eqref{EQUATSS2.3GradiLogar_Rho_} is a gradient, so we deduce that the Jacobian matrix of this quantity is symmetric. On the other hand replacing $u$ by its expression \eqref{EQUATSS2.2PremiExpre__u__}, we find:
\begin{align}
\label{EQUATSS2.3_grad_log__rho_}
\nabla_x \left(\ln\rho\right) = M x - 2\sigma_0 \Lambda C - \sigma_0' C - 2 \sigma_0 C'
\end{align}
with
\begin{align}
\label{EQUATSS2.3DefinMatri__M__}
M = - 2 \Big[ \sigma_0 \Lambda^2 + \sigma_0' \Lambda + \sigma_0 \Lambda' \Big] - 2 \sigma_0 \left[ \left( \frac{\sigma_0'}{2\sigma_0} \right)^2 + \left(\frac{\sigma_0'}{2\sigma_0}\right)' \right] I_d.
\end{align}
We deduce therefore that the skew-symmetric part of the matrix $M$ has to be zero, that is we have the equation:
\begin{align}
\label{EQUATSS2.3Equat_Lambda(t)}
\sigma_0' \Lambda + \sigma_0 \Lambda' = 0,
\end{align}
or again, we find that $\sigma_0 \Lambda$ is a time-independent matrix. Therefore, we obtain another expression for $u$, summarized in the following lemma.

\begin{lemma}[Second expression of the velocity field $u$]
\label{LEMMESS2.3DeuxiExpre__u__}
Let $m : [0,T] \times \Omega \times \mathbb{R}^d \rightarrow \mathbb{R}_+$ be a local Maxwellian of the form:
\begin{align*}
m(t,x,v) = \rho(t,x) \exp \left( -a(t,x) \left\vert v - u(t,x) \right\vert^2 \right),
\end{align*}
with $\rho$ and $a$ differentiable and $u$ twice differentiable in $x$, with $\rho$ and $a$ positive almost everywhere, that solves on $]0,T[\times\Omega\times\mathbb{R}^d$ the free transport equation:
\begin{align*}
\partial_t m + v \cdot\nabla_x m = 0.
\end{align*}
Then, there exist a non-negative real-valued, differentiable function $t \mapsto \sigma_0(t) \geq 0$, strictly positive almost everywhere, which coincides with the function given in Corollary \ref{COROLSS2.1PremiSimpl__a__}, a skew-symmetric matrix $\Lambda_0 \in \mathcal{M}_d(\mathbb{R})$, and a time-dependent vector $t \mapsto C(t) \in \mathbb{R}^d$ such that the velocity field $u$ can be written as:
\begin{align}
\label{EQUATSS2.3DeuxiExpre__u__}
u(t,x) = \frac{1}{\sigma_0(t)} \Lambda_0 x + \left( \frac{\sigma_0'(t)}{2\sigma_0(t)} \right) x + C(t).
\end{align}
\end{lemma}
\noindent
Let us now use the previous result \eqref{EQUATSS2.3Equat_Lambda(t)} on the matrix $\Lambda(t)$ to obtain a first expression of the density $\rho$. Since $\ln\rho$ satisfies the equation \eqref{EQUATSS2.3_grad_log__rho_} with $M$ given by \eqref{EQUATSS2.3DefinMatri__M__}, we deduce the following result.

\begin{lemma}[First expression of the density $\rho$]
\label{LEMMESS2.3SPremiExpre_rho_}
Let $m : [0,T] \times \Omega \times \mathbb{R}^d \rightarrow \mathbb{R}_+$ be a local Maxwellian of the form:
\begin{align*}
m(t,x,v) = \rho(t,x) \exp \left( -a(t,x) \left\vert v - u(t,x) \right\vert^2 \right),
\end{align*}
with $\rho$, $a$ and $u$ differentiable and $u$ twice differentiable in $x$, with $\rho$ and $a$ positive almost everywhere, that solves on $]0,T[\times\Omega\times\mathbb{R}^d$ the free transport equation:
\begin{align*}
\partial_t m + v \cdot\nabla_x m = 0.
\end{align*}
Then, there exist two non-negative real-valued, differentiable function $t \mapsto \rho_0(t) \geq 0$ and $t \mapsto \sigma_0(t) \geq 0$, strictly positive almost everywhere, which coincides with the function given in Corollary \ref{COROLSS2.1PremiSimpl__a__}, a skew-symmetric matrix $\Lambda_0 \in \mathcal{M}_d(\mathbb{R})$, and a time-dependent vector $t \mapsto C(t) \in \mathbb{R}^d$ such that the density $\rho$ can be written as:
\begin{align}
\label{EQUATSS2.3PremiExpre_rho_}
\rho(t,x) &= \rho_0(t) \exp \Big( (-x) \cdot \left( \left[ \frac{1}{\sigma_0(t)} \Lambda_0^2 + \sigma_0(t) \left( \left(\frac{\sigma_0'(t)}{2\sigma_0(t)}\right)^2 + \left(\frac{\sigma_0'(t)}{2\sigma_0(t)}\right)' \right) I_d \right] x \right) \nonumber\\
&\hspace{45mm} - 2\left(\Lambda_0 C(t) \right) \cdot x - \sigma_0'(t)( C(t) \cdot x ) - 2\sigma_0(t) ( C'(t) \cdot x ) \Big)
\end{align}
\end{lemma}

\subsection{The general expressions of the coefficients $\sigma_0$, $C$ and $\rho_0$}

We can now determine completely the functions $\sigma_0$, $C$ and $\rho_0$. To do so, let us replace in the equation \eqref{EQUATSS2.1Systemerho_a_u5} the functions $\rho$ and $u$ by the expressions \eqref{EQUATSS2.3PremiExpre_rho_} and \eqref{EQUATSS2.3DeuxiExpre__u__} we found in Lemmas \ref{LEMMESS2.3SPremiExpre_rho_} and \ref{LEMMESS2.3DeuxiExpre__u__}. Dividing by the exponential term, we find:

\begin{align}
\label{EQUATSS2.4Equa1TransRho_u}
0 &= \rho_0' + \rho_0 \Big( -x \cdot \left( \left[ -\frac{\sigma_0'}{\sigma_0^2} \Lambda_0^2 + \left(\sigma_0\phi\right)' I_d \right] x \right) - 2 \left(\left(\Lambda_0 C'\right)\cdot x\right) - \sigma_0'' (C \cdot x) - 3\sigma_0' (C' \cdot x) - 2\sigma_0(C'' \cdot x) \Big) \nonumber\\
&\hspace{5mm} + \rho_0 \Big( \frac{1}{\sigma_0}\Lambda_0 x + \left(\frac{\sigma_0'}{2\sigma_0}\right) x + C \Big) \cdot \Big( -2 \left[ \frac{1}{\sigma_0}\Lambda_0^2 + \sigma_0 \phi I_d \right] x -2 (\Lambda_0 C) - \sigma_0' C - 2 \sigma_0 C'\Big),
\end{align}
where $\phi$ denotes the function:
\begin{align}
\label{EQUATSS2.4DefinFonct_phi_}
\phi(t) = \left( \frac{\sigma_0'(t)}{2\sigma_0(t)} \right)^2 + \left( \frac{\sigma_0'(t)}{2\sigma_0(t)} \right)'.
\end{align}
Observing now that \eqref{EQUATSS2.4Equa1TransRho_u} is a polynomial equation in $x$, with coefficients that are time-dependent functions, we can rearrange the terms and regroup them according to their respective powers in $x$, which provides:
\begin{align}
\label{EQUATSS2.4Equa2TransRho_u}
0 &= \rho_0 \left[ - \left(\sigma_0\phi\right)' \vert x \vert^2 - \sigma_0'\phi \vert x \vert^2 \right] \nonumber\\
&\hspace{5mm} +\rho_0 \left[ -\sigma_0'' C\cdot x - 4\sigma_0'C'\cdot x - 2\sigma_0 C''\cdot x - \frac{(\sigma_0')^2}{2\sigma_0} C\cdot x - 2\sigma_0 \phi C \cdot x \right] \nonumber\\
&\hspace{5mm} + \rho_0' + \rho_0 \left[ -\sigma_0' \vert C \vert^2 - 2 \sigma_0 C'\cdot C \right]
\end{align}
using in particular that $\Lambda_0$ is skew-symmetric, so that $~^tx \Lambda_0 x = 0$ and $~^tx \Lambda_0^3 x = 0$. Therefore, from \eqref{EQUATSS2.4Equa2TransRho_u} we deduce the system:
\begin{subequations}
\label{EQUATSS2.4SysteC_Lambda0_Sigm0}
\begin{empheq}[left=\empheqlbrace]{align}
\left( \sigma_0 \phi \right)' + \sigma_0' \phi = 2\sigma_0'\phi + \sigma_0 \phi' &= 0, \label{EQUATSS2.4SysteC_Lambda0_Sigm1}\\
2\sigma_0 C'' + 4\sigma_0' C' + \left( \sigma_0'' + \frac{(\sigma_0')^2}{2\sigma_0} + 2\sigma_0 \phi \right) C &= 0, \label{EQUATSS2.4SysteC_Lambda0_Sigm2}\\
\rho_0' + \rho_0 \left( -\sigma_0' \vert C \vert^2 - 2\sigma_0 C' \cdot C \right)&= 0, \label{EQUATSS2.4SysteC_Lambda0_Sigm3}
\end{empheq}
\end{subequations}
with $\phi$ given by \eqref{EQUATSS2.4DefinFonct_phi_}.\\
The final step of the proof of Theorem \ref{THEORSS2.1ExpreMaxweLocGn} is to solve the system \eqref{EQUATSS2.4SysteC_Lambda0_Sigm0}, that is, to determine completely the expressions of the functions $\sigma_0$, $C$ and $\rho_0$.

\paragraph{Determining $\sigma_0$.}

To determine $\sigma_0$, without using the explicit expression of $\phi$ for the moment, let us observe that \eqref{EQUATSS2.4SysteC_Lambda0_Sigm1} provides the expression of a derivative when multiplied by $\sigma_0$, namely:
\begin{align*}
2 \sigma_0' \sigma_0 \phi + \sigma_0^2 \phi' = \frac{\dd}{\dd t} \left( \sigma_0^2 \phi \right) = 0,
\end{align*}
so that we deduce that the function
\begin{align*}
t \mapsto \sigma_0^2(t) \phi(t)
\end{align*}
is constant. But since we have
\begin{align}
\phi = \frac{1}{4 \sigma_0^2} \left( 2\sigma_0'' \sigma_0 - (\sigma_0')^2 \right),
\end{align}
we obtain that the function
\begin{align*}
t \mapsto 2\sigma_0''(t) \sigma_0(t) - \left( \sigma_0'(t) \right)^2
\end{align*}
is equal to some constant $k \in \mathbb{R}$. Differentiating once again the expression of this last function, we find:
\begin{align}
2 \sigma_0''' \sigma_0 + 2 \sigma_0'' \sigma_0' - 2\sigma_0'' \sigma_0' = 2 \sigma_0''' \sigma_0 = 0.
\end{align}
Keeping finally in mind that $\sigma_0$ is almost everywhere positive, we deduce that $\sigma_0'''$ is constant, equal to zero, up to assume that $\sigma_0'''$ is well defined and continuous. We can then obtain a complete description of the function $\sigma_0$.

\begin{lemma}[Complete expression of the coefficient $a$]
\label{LEMMESS2.4ExpreCompl__a__}
Let $m : [0,T] \times \Omega \times \mathbb{R}^d \rightarrow \mathbb{R}_+$ be a local Maxwellian of the form:
\begin{align*}
m(t,x,v) = \rho(t,x) \exp \left( -a(t,x) \left\vert v - u(t,x) \right\vert^2 \right),
\end{align*}
with $\rho$, $a$ and $u$ differentiable and $u$ twice differentiable in $x$, with $\rho$ and $a$ positive almost everywhere, that solves on $]0,T[\times\Omega\times\mathbb{R}^d$ the free transport equation:
\begin{align*}
\partial_t m + v \cdot\nabla_x m = 0.
\end{align*}
Then, there exist three real numbers $\alpha$, $\beta$ and $\gamma \in \mathbb{R}$ such that the coefficient $a$ satisfies:
\begin{align}
\label{EQUATSS2.4ExpreCompl__a__}
\forall\, (t,x) \in [0,T]\times\Omega, \hspace{5mm} a(t,x) = \alpha t^2 + \beta t + \gamma.
\end{align}
\end{lemma}

\paragraph{Determining $C$.}

Turning now to the determination of $C$, we start with rewritting \eqref{EQUATSS2.4SysteC_Lambda0_Sigm2}, using the result of Lemma \ref{LEMMESS2.4ExpreCompl__a__}. We find:
\begin{align*}
\sigma_0 C'' + 2 \sigma_0'C' + \sigma_0'' C = \left( \alpha t^2 + \beta t + \gamma \right) C'' + 2 \left( 2\alpha t + \beta \right) C' + 2 \alpha C = 0.
\end{align*}
In particular, \eqref{EQUATSS2.4SysteC_Lambda0_Sigm2} can be rewritten as:
\begin{align*}
\frac{\dd^2}{\dd t^2} \left[ \left( \alpha t^2 + \beta t + \gamma \right) C(t) \right] = 0.
\end{align*}
This last equation leads to the following result.

\begin{lemma}[Complete expression of the vector $C$]
\label{LEMMESS2.4ExpreCompl__C__}
Let $m : [0,T] \times \Omega \times \mathbb{R}^d \rightarrow \mathbb{R}_+$ be a local Maxwellian of the form:
\begin{align*}
m(t,x,v) = \rho(t,x) \exp \left( -a(t,x) \left\vert v - u(t,x) \right\vert^2 \right),
\end{align*}
with $\rho$, $a$ and $u$ differentiable and $u$ twice differentiable in $x$, with $\rho$ and $a$ positive almost everywhere, that solves on $]0,T[\times\Omega\times\mathbb{R}^d$ the free transport equation:
\begin{align*}
\partial_t m + v \cdot\nabla_x m = 0.
\end{align*}
Then, there exist two vectors $w_1$, $w_2 \in \mathbb{R}^d$ such that the vector $C$ defined by \eqref{EQUATSS2.3DeuxiExpre__u__} can be written as:
\begin{align}
\label{EQUATSS2.4ExpreCompl__C__}
\forall\, t \in [0,T], \hspace{5mm} C(t) = \frac{tw_1 + w_2}{\alpha t^2 + \beta t + \gamma},
\end{align}
where $\alpha$, $\beta$ and $\gamma \in \mathbb{R}$ are given in Lemma \ref{LEMMESS2.4ExpreCompl__a__}.
\end{lemma}

\paragraph{Determining $\rho_0$.}

To solve completely the system \eqref{EQUATSS2.4SysteC_Lambda0_Sigm0}, it remains to determine $\rho_0$. This can be done in a fairly straightforward way, by noticing that \eqref{EQUATSS2.4SysteC_Lambda0_Sigm3} can be written as:
\begin{align*}
\rho_0' + \frac{\dd}{\dd t}\left( -\sigma_0 \vert C \vert^2 \right) \rho_0 = 0.
\end{align*}
We deduce that:
\begin{align}
\label{EQUATSS2.4PremiExpre_rho0}
\rho_0(t) = r_0 \exp \left( \sigma_0 \vert C \vert^2 \right).
\end{align}
Such a computation provides the following result, which concludes the study of the system \eqref{EQUATSS2.4SysteC_Lambda0_Sigm0}.

\begin{lemma}[Complete expression of the coefficient $\rho_0$]
\label{LEMMESS2.4ExpreCompl_rho0}
Let $m : [0,T] \times \Omega \times \mathbb{R}^d \rightarrow \mathbb{R}_+$ be a local Maxwellian of the form:
\begin{align*}
m(t,x,v) = \rho(t,x) \exp \left( -a(t,x) \left\vert v - u(t,x) \right\vert^2 \right),
\end{align*}
with $\rho$, $a$ and $u$ differentiable and $u$ twice differentiable in $x$, with $\rho$ and $a$ positive almost everywhere, that solves on $]0,T[\times\Omega\times\mathbb{R}^d$ the free transport equation:
\begin{align*}
\partial_t m + v \cdot\nabla_x m = 0.
\end{align*}
Then, there exists a positive real number $r_0 \in \mathbb{R}_+^*$ such that $\rho_0$ defined by \eqref{EQUATSS2.3PremiExpre_rho_} can be written as:
\begin{align}
\forall\, t \in [0,T], \hspace{5mm} \rho_0(t) = r_0 \exp\left( \frac{\vert tw_1 + w_2 \vert^2}{\alpha t^2 + \beta t + \gamma} \right),
\end{align}
where $\alpha$, $\beta$ and $\gamma \in \mathbb{R}$ are given in Lemma \ref{LEMMESS2.4ExpreCompl__a__}, and $w_1$, $w_2 \in \mathbb{R}^d$ are given in Lemma \ref{LEMMESS2.4ExpreCompl__C__}.
\end{lemma}

\subsection{Concluding the proof concerning the expression of the local Maxwellians}

We are now in position to give the expression of a general local Maxwellian that solves the free transport equation \eqref{EQUATSect2Transport_Libre}.

\begin{proof}[Proof of Theorem \ref{THEORSS2.1ExpreMaxweLocGn}]
Considering a local Maxwellian, given by an expression of the form \eqref{EQUATSect2ExpreMaxweLocal}, we start with replacing $\rho$ using \eqref{EQUATSS2.3PremiExpre_rho_} and \eqref{EQUATSS2.4PremiExpre_rho0}, $a$ using \eqref{EQUATSS2.1PremiExpre__a__}, and $u$ using \eqref{EQUATSS2.3DeuxiExpre__u__}. We find:
\begin{align}
m(t,x,v) &= \rho(t,x) \exp \left( -a(t,x) \vert v - u(t,x) \vert^2 \right) \nonumber\\
&= \rho_0(t) \exp \Bigg( - \sigma_0 \left[ x\cdot\left( \frac{\Lambda_0^2}{\sigma_0^2} x \right) + \left( \frac{\sigma_0'}{2\sigma_0} \right)^2 \vert x \vert^2 + \left( \frac{\sigma_0'}{2\sigma_0} \right)' \vert x \vert^2 \right] \nonumber\\
&\hspace{80mm} - 2 \left( \Lambda_0 C \right) \cdot x - \sigma_0' C\cdot x - 2 \sigma_0 C'\cdot x \Bigg) \nonumber\\
&\hspace{5mm} \cdot \exp \left( - \sigma_0 \left\vert v - \frac{1}{\sigma_0}\Lambda_0 x - \left( \frac{\sigma_0'}{2\sigma_0} \right) x - C \right\vert^2 \right) \nonumber\\
&= r_0 \exp\left(\sigma_0 \vert C \vert^2 \right) \exp \Bigg( - \sigma_0 \left[ x\cdot\left( \frac{\Lambda_0^2}{\sigma_0^2} x \right) + \left( \frac{\sigma_0'}{2\sigma_0} \right)^2 \vert x \vert^2 + \left( \frac{\sigma_0'}{2\sigma_0} \right)' \vert x \vert^2 \right] \nonumber\\
&\hspace{80mm} - 2 \left( \Lambda_0 C \right) \cdot x - \sigma_0' C\cdot x - 2 \sigma_0 C'\cdot x \Bigg) \nonumber\\
&\hspace{5mm} \cdot \exp \left( - \sigma_0 \left\vert v - \frac{1}{\sigma_0}\Lambda_0 x - \left( \frac{\sigma_0'}{2\sigma_0} \right) x - C \right\vert^2 \right).
\end{align}
Simplifications take place, which provides:
\begin{align}
m(t,x,v) &= r_0 \exp \Bigg( - \sigma_0 \left[ x\cdot\left( \frac{\Lambda_0^2}{\sigma_0^2} x \right) + \left( \frac{\sigma_0'}{2\sigma_0} \right)^2 \vert x \vert^2 + \left( \frac{\sigma_0'}{2\sigma_0} \right)' \vert x \vert^2 \right] \nonumber\\
&\hspace{80mm} - 2 \left( \Lambda_0 C \right) \cdot x - \sigma_0' C\cdot x - 2 \sigma_0 C'\cdot x \Bigg) \nonumber\\
&\hspace{5mm} \cdot \exp \left( -\sigma_0 \left\vert v - \frac{1}{\sigma_0}\Lambda_0 x - \left( \frac{\sigma_0'}{2\sigma_0} \right) x \right\vert^2 + 2 \sigma_0 \left(v - \frac{1}{\sigma_0}\Lambda_0 x - \left(\frac{\sigma_0'}{2\sigma_0}\right) x \right)\cdot C \right), \nonumber\\
&= r_0 \exp \Bigg( - \sigma_0 \left[ x\cdot\left( \frac{\Lambda_0^2}{\sigma_0^2} x \right) + \left( \frac{\sigma_0'}{2\sigma_0} \right)^2 \vert x \vert^2 + \left( \frac{\sigma_0'}{2\sigma_0} \right)' \vert x \vert^2 \right] - \sigma_0' C\cdot x - 2 \sigma_0 C'\cdot x \Bigg) \nonumber\\
&\hspace{40mm} \cdot \exp \left( -\sigma_0 \left\vert v - \frac{1}{\sigma_0}\Lambda_0 x - \left( \frac{\sigma_0'}{2\sigma_0} \right) x \right\vert^2 + 2 \sigma_0 \left(v - \left(\frac{\sigma_0'}{2\sigma_0}\right) x \right)\cdot C \right), \nonumber\\
&= r_0 \exp \Bigg( - \sigma_0 \left[ x\cdot\left( \frac{\Lambda_0^2}{\sigma_0^2} x \right) + \left( \frac{\sigma_0'}{2\sigma_0} \right)^2 \vert x \vert^2 + \left( \frac{\sigma_0'}{2\sigma_0} \right)' \vert x \vert^2 \right] - 2 \sigma_0' C\cdot x - 2 \sigma_0 C'\cdot x \Bigg) \nonumber\\
&\hspace{40mm} \cdot \exp \left( -\sigma_0 \left\vert v - \frac{1}{\sigma_0}\Lambda_0 x - \left( \frac{\sigma_0'}{2\sigma_0} \right) x \right\vert^2 + 2 \sigma_0 C \cdot v \right).
\end{align}
using in particular that
\begin{align*}
-2\left( \Lambda_0 C\right)\cdot x - 2 \left(\Lambda_0 x\right)\cdot C = -2 ~^tC \left(~^t\Lambda_0\right) x - 2~^tC \Lambda_0 x = 0.
\end{align*}
Observing now that:
\begin{align*}
\sigma_0' C + \sigma_0 C' &= \frac{1}{\sigma_0} \left[ 2\alpha t^2 w_1 + 2 \alpha t w_2 + \beta t w_1 + \beta w_2 + \left(\alpha t^2 + \beta t + \gamma \right) w_1 - \left(tw_1+w_2\right) \left(2\alpha t + \beta\right) \right] \\
&= \frac{1}{\sigma_0} \left[ \alpha t^2 w_1 + \beta t w_1 + \gamma w_1 \right] = w_1,
\end{align*}
we find:
\begin{align}
m(t,x,v) &= r_0 \exp \left( -\sigma_0 \left\vert v - \frac{1}{\sigma_0}\Lambda_0 x - \left( \frac{\sigma_0'}{2\sigma_0} \right) x \right\vert^2 + 2 \sigma_0 C \cdot v \right) \nonumber\\
&\hspace{25mm} \cdot \exp \Bigg( - \sigma_0 \left[ x\cdot\left( \frac{\Lambda_0^2}{\sigma_0^2} x \right) + \left( \frac{\sigma_0'}{2\sigma_0} \right)^2 \vert x \vert^2 + \left( \frac{\sigma_0'}{2\sigma_0} \right)' \vert x \vert^2 \right] - 2 w_1\cdot x \Bigg)
\end{align}
Now, developing the first term in the exponential, we obtain:
\begin{align*}
\left\vert v - \frac{1}{\sigma_0} \Lambda_0 x - \left(\frac{\sigma_0'}{2\sigma_0}\right) x \right\vert^2 &= \vert v \vert^2 + \frac{1}{\sigma_0^2} \left( \Lambda_0 x \right)\cdot\left(\Lambda_0 x\right) + \left(\frac{\sigma_0'}{2\sigma_0}\right)^2 \vert x \vert^2 \\
&\hspace{25mm} - \frac{2}{\sigma_0} \left( \Lambda_0 x\right)\cdot v - \left(\frac{\sigma_0'}{\sigma_0}\right) x\cdot v + \underbrace{\frac{\sigma_0'}{\sigma_0^2}\left(\Lambda_0 x\right)\cdot x}_{=0},
\end{align*}
so that
\begin{align}
m(t,x,v) &= r_0 \exp \left( \hspace{-1mm} - \sigma_0 \vert v \vert^2 - \hspace{-0.5mm} 2 \sigma_0 \left( \frac{\sigma_0'}{2\sigma_0} \right)^2 \vert x \vert^2 + \hspace{-0.5mm} 2\left(\Lambda_0 x\right)\cdot v + \hspace{-0.5mm} \sigma_0' x\cdot v \hspace{-0.5mm} + 2\sigma_0 C \cdot v \hspace{-0.5mm} - \sigma_0 \left(\frac{\sigma_0'}{2\sigma_0}\right)' \vert x \vert^2 \hspace{-0.5mm} - 2w_1\cdot x \hspace{-1mm}\right),
\end{align}
because we have in particular:
\begin{align*}
-\frac{1}{\sigma_0} \left( \Lambda_0 x \right) \cdot \left( \Lambda_0 x \right) - \sigma_0 x \cdot \left( \frac{\Lambda_0^2}{\sigma_0^2} x \right) = - \frac{1}{\sigma_0} (~^tx) \left(~^t\Lambda_0\right) \Lambda_0 x - \frac{1}{\sigma_0} ~^t x \Lambda_0^2 x = 0.
\end{align*}
Now, let us note that:
\begin{align*}
2\sigma_0 \left( \frac{\sigma_0'}{2\sigma_0} \right)^2 + \sigma_0 \left( \frac{\sigma_0'}{2\sigma_0} \right)' &= \sigma_0 \left[ \frac{8\alpha^2t^2+8\alpha\beta t + 2\beta^2 - 4\alpha^2t^2 - 4\alpha\beta t + 4\alpha\gamma - 2\beta^2}{(2\sigma_0)^2} \right] \\
&= \alpha,
\end{align*}
so that we have:
\begin{align}
m(t,x,v) &= r_0 \exp\left( -\sigma_0\vert v \vert^2 -\alpha \vert x \vert^2 + 2 \left(\Lambda_0 x\right)\cdot v + \sigma_0' x\cdot v + 2 \sigma_0 C\cdot v - 2w_1\cdot x \right).
\end{align}
Finally, using \eqref{EQUATSS2.4ExpreCompl__a__} and \eqref{EQUATSS2.4ExpreCompl__C__} we have:
\begin{align*}
-\sigma_0 \vert v \vert^2 - \alpha \vert x \vert^2 + \sigma_0' x \cdot v + 2 \sigma_0 C \cdot v - 2 w_1 \cdot x &= -\left( \alpha t^2 + \beta t + \gamma \right) \vert v \vert^2 - \alpha \vert x \vert^2 \\
&\hspace{10mm} + \left( 2\alpha t +\beta \right) x \cdot v + 2(w_1t+w_2) \cdot v - 2 w_1 \cdot x \\
&= - \alpha \vert x-tv \vert^2 + \beta v \cdot (x-tv) - \gamma \vert v \vert^2 - 2 w_1\cdot(x-tv) + 2 w_2\cdot v.
\end{align*}
It is therefore possible to deduce that if a local Maxwellian solves the free transport equation, it can be written in the form given by \eqref{EQUATSS2.1ExpreMaxweLocGn}, which concludes the proof of Theorem \ref{THEORSS2.1ExpreMaxweLocGn}.
\end{proof}

\noindent
Thanks to Theorem \ref{THEORSS2.1ExpreMaxweLocGn}, we can now describe completely the local Maxwellians that solve the Boltzmann equation, inside particular domains in $\mathbb{R}^2$ or $\mathbb{R}^3$ (even $\mathbb{R}^d$ in the bounce-back case), depending also on the choice of the boundary condition we prescribe. This is the object of the two following sections.

\section{Local Maxwellians and boundary condition I: the bounce-back case}

Let us start with the case of the \emph{bounce-back boundary condition}. We consider a gas evolving in a regular open set $\Omega$ of $\mathbb{R}^d$. This boundary condition, which may not look natural at the first glance, is defined as follows.

\begin{defin}[Bounce-back boundary condition (BBBC)]
Let $\Omega$ be a regular, open set in $\mathbb{R}^d$, and let $f:\overline{\Omega} \rightarrow \mathbb{R}$ be any function. $f$ is said to satisfy the \emph{bounce-back boundary condition} (abbreviated as \emph{BBBC}) on the boundary $\partial \Omega$ of $\Omega$ if:
\begin{align}
\label{EQUATSect3BounceBack_B_C_}
\forall t \in [0,T],\, x \in \partial \Omega,\, v \in \mathbb{R}^d,\hspace{3mm} f(t,x,v) = f(t,x,-v).
\end{align}
\end{defin}
\noindent
At the level of the particles, this means that the particles are reflected on the boundary of the domain $\Omega$, in such a way that the particles travel back, following backwards exactly the path they took before they collided with the wall.\\
\newline
Let us now characterize the local Maxwellians solving the free transport equation with bounce-back boundary condition.

\begin{theor}[Characterization of the local Maxwellians solving the free transport equation in a domain with bounce-back boundary condition (BBBC)]
\label{THEORSect3BounceBack_B_C_}
Let $T$ be a strictly positive number (possibly $+\infty$), let $\Omega$ be a regular open set of $\mathbb{R}^d$ with a non empty boundary, and let us consider a local Maxwellian $m$, that is, a function of the form \eqref{EQUATSS2.1ExpreMaxweLocGn}, which solves the transport equation \eqref{EQUATSect2Transport_Libre} on $[0,T]\times\Omega\times\mathbb{R}^d$, with the bounce-back boundary condition, that is, such that
\begin{align}
\forall t \in [0,T], x \in \partial \Omega, v \in \mathbb{R}^d,\hspace{3mm} m(t,x,v) = m(t,x,-v).
\end{align}
Then, there exist two real numbers $r_0$ and $\gamma \in \mathbb{R}$ such that:
\begin{align}
\label{EQUATSect3MaxwlBounceBack}
\forall t \in [0,T], x \in \Omega, v \in \mathbb{R}^d, \hspace{3mm} m(t,x,v) = r_0 \exp\left( -\gamma \vert v \vert^2 \right) \hspace{5mm} .
\end{align}
\end{theor}

\begin{remar}
\eqref{EQUATSect3MaxwlBounceBack} shows in particular that for a mass and a kinetic energy fixed, there exists only one local Maxwellian that solves the free transport equation with bounce-back boundary condition. Such a Maxwellian has zero bulk (the momentum of $m$ is zero), and it is actually global (none of its coefficient depends on the time $t$ nor on the position variable $x$).
\end{remar}

\begin{remar}
In the literature, it is usual to find the assumption that the domain $\Omega$ is bounded (see for instance \cite{Desv990}, \cite{DeVi005} or \cite{ReVi008}). This assumption simplifies the proof of Theorem \ref{THEORSect3BounceBack_B_C_}, but it is not necessary, in the sense that it is possible to obtain the same result, without any assumption concerning the boundedness of the domain.
\end{remar}

\begin{proof}[Proof of Theorem \ref{THEORSect3BounceBack_B_C_}]
Writing the local Maxwellian $m$ in its general form \eqref{EQUATSect2Transport_Libre}, the bounce-back boundary condition \eqref{EQUATSect3BounceBack_B_C_} writes:
\begin{align*}
\rho(t,x) \exp \left( - a(t,x) \left\vert v - u(t,x) \right\vert^2 \right) = \rho(t,x) \exp \left( - a(t,x) \left\vert -v - u(t,x) \right\vert^2 \right)
\end{align*}
for all $t \in [0,T]$, $x \in \partial \Omega$, and $v \in \mathbb{R}^d$. This implies that $u(t,x)$ has to be zero for all time $t$, and all $x$ that belongs to the boundary $\partial \Omega$ of the domain.\\
Remembering now that we obtained, along the proof of Theorem \ref{THEORSS2.1ExpreMaxweLocGn}, that $u(t,x)$ writes:
\begin{align*}
u(t,x) = \frac{1}{\sigma_0(t)} \Lambda_0 x + \frac{\sigma_0'(t)}{2\sigma_0(t)}x + \frac{tw_1 + w_2}{\sigma_0(t)},
\end{align*}
(see in particular \eqref{EQUATSS2.3DeuxiExpre__u__} page \pageref{EQUATSS2.3DeuxiExpre__u__} and \eqref{EQUATSS2.4ExpreCompl__C__} page \pageref{EQUATSS2.4ExpreCompl__C__}), and using the fact that $\sigma_0(t)$ was the polynomial $\alpha t^2 + \beta t + \gamma$ of degree $2$ in $t$ (see \eqref{EQUATSS2.4ExpreCompl__a__} page \pageref{EQUATSS2.4ExpreCompl__a__}), we obtain:
\begin{align}
\label{EQUATSect3CondiGener__u__}
\Lambda_0 x + (\alpha t + \beta/2) x + tw_1 + w_2 = 0,
\end{align}
which implies that the first order coefficient (in $t$) in \eqref{EQUATSect3CondiGener__u__} has to vanish, that is, we have:
\begin{align}
\label{EQUATChptrLgTBhBounceBackAlpha}
\alpha x + w_1 = 0
\end{align}
for all point $x \in \partial \Omega$. If $\alpha \neq 0$, \eqref{EQUATChptrLgTBhBounceBackAlpha} cannot hold, because $\partial \Omega$ was assumed to be non empty, and so it cannot be reduced to a single point by regularity. Therefore, we have $\alpha = 0$.\\
Using then the expression \eqref{EQUATSS2.1ExpreMaxweLocGn} and applying again the boundary condition, we obtain the new equation:
\begin{align*}
4 v \cdot \left( \Lambda_0 x \right) + 4 v \cdot w_2 + 4 t v\cdot w_1 + 2 \beta v \cdot x = 0,
\end{align*}
for all $t \in [0,T]$, $x \in \partial \Omega$ and $v \in \mathbb{R}^d$. We deduce immediately the new condition:
\begin{align}
2 \Lambda_0 x + 2 w_2 + 2 t w_1 + \beta x = 0,
\end{align}
for all $t \in [0,T]$, $x \in \partial \Omega$. The first order term in $t$ provides directly that
\begin{align}
\label{EQUATChptrLgTBhBounceBackZero1}
w_1 = 0,
\end{align}
while the constant term gives:
\begin{align}
\label{EQUATChptrLgTBhBounceBackCond1}
2 \Lambda_0 x + \beta x + 2 w_2 = 0,
\end{align}
everywhere on the boundary $\partial \Omega$ of the domain. Let us now notice that \eqref{EQUATChptrLgTBhBounceBackCond1} cannot hold if $\beta \neq 0$: indeed, in that case, we would have that the mapping $\phi: x \mapsto 2 \Lambda_0 x + \beta x$ is invertible because %\footnote{Another proof of the invertibility, suggested by Bernhard Kepka, is the following: the spectrum of $2\Lambda_0$ is contained in $i \mathbb{R}$ (since it is a skew-symmetric matrix), so that the spectrum of the matrix $2\Lambda_0 + \beta I_d$ (which is simply the spectrum of $2\Lambda_0$, shifted along the horizontal axis) never crosses the origin: the matrix is therefore invertible.} 
for any non-zero vector $x$, we find:
\begin{align*}
x \cdot \phi(x) = 2x \cdot \left( \Lambda_0 x \right) + \beta \vert x \vert^2 = \beta \vert x \vert^2 \neq 0.
\end{align*}
Therefore, we have:
\begin{align}
\label{EQUATChptrLgTBhBounceBackZero2}
\beta = 0.
\end{align}
But then, \eqref{EQUATChptrLgTBhBounceBackCond1} can be rewritten as:
\begin{align}
\label{EQUATChptrLgTBhBounceBackCond2}
\Lambda_0 x + w_2 = 0,
\end{align}
for all $x \in \partial \Omega$. Let us show that if $\Lambda_0 \neq 0$, we obtain a contradiction. In that case, the matrix $\Lambda_0$ is not trivial, and so its kernel cannot be of dimension $d$. Let us fix a point $x_0 \in \partial \Omega$. In particular, we have
\begin{align*}
\Lambda_0 x_0 + w_2 = 0,
\end{align*}
and so, for any other point $x \in \partial \Omega$, we have by linearity:
\begin{align*}
\Lambda_0 (x-x_0) = 0,
\end{align*}
that is, $x-x_0$ is contained in the kernel of $\Lambda_0$, or again:
\begin{align}
\label{EQUATChptrLgTBhBounceBackCond3}
\partial \Omega \subset x_0 + \text{Ker}\Lambda_0.
\end{align}
This last equation shows already that if the boundary $\partial \Omega$ is not an affine hyperplane of $\mathbb{R}^d$, then we cannot have $\Lambda_0 \neq 0$. Let us show that even when the boundary of the domain $\Omega$ is flat, we obtain a contradiction anyway, due to the dimension of the kernel of $\Lambda_0$.\\
\newline
Let us assume that $\Lambda_0$ is not the zero matrix. We will prove that the dimension of the kernel of $\Lambda_0$ is at most $d-2$.\\
The assumption that $\Lambda_0 \neq 0$ implies that its spectrum cannot be reduced to zero. Indeed, a skew-symmetric matrix is a particular case of a normal matrix (that is, it commutes with its Hermitian conjugate $\Lambda_0^*$), so there exists a unitary matrix $U$ and a diagonal matrix $D$ (both with complex entries) such that: such that
\begin{align*}
\Lambda_0 = U D U^* \hspace{5mm} \text{and} \hspace{5mm} U U^* = I_d.
\end{align*}
Therefore, there exists a non-zero eigenvector $z \in \mathbb{C}^d$ of $\Lambda_0$ associated to a non-zero eigenvalue $\lambda \in \mathbb{C}$. Let us decompose $z$ into its real and imaginary parts:
\begin{align*}
z = u_1 + i u_2 \hspace{5mm} \text{with} \hspace{5mm} u_1,u_2 \in \mathbb{R}^d.
\end{align*}
$\Lambda_0$ being skew-symmetric, we know in addition that $\lambda$ has to be purely imaginary. Besides, $z$ being non zero, either its real part $u_1$ or its imaginary part $u_2$ is non-zero. Actually both $u_1$ and $u_2$ are non-zero. Indeed, if $u_2$ would be zero, $z$ would be a real-valued vector, which would lead to a contradiction, because $\Lambda_0$ is also real-valued and $\lambda \in i \mathbb{R}$. In the same way, $u_1$ cannot be zero.\\
Therefore, the real and imaginary parts $u_1$ and $u_2$ of the eigenvector $z$ are both non zero. They are in addition linearly independent in $\mathbb{R}^d$. Let us indeed assume that there exists a real number $\mu$ (necessarily non zero) such that $u_2 = \mu u_1$. Using that $\Lambda_0$ is a real matrix, we obtain another eigenvector, associated to a different eigenvalue, considering simply the eigenequation:
\begin{align*}
\overline{ \left(\Lambda_0 \cdot z\right) } &= \Lambda_0 \cdot \overline{z} \\
& = \overline{\left( \lambda z \right)} = \overline{\lambda} \overline{z}.
\end{align*}
We would have then:
\begin{align*}
z = u_1+iu_2 = (1 + i \mu) u_1, \text{   and   } \overline{z} = u_1 - i u_2 = (1-i\mu)u_1.
\end{align*}
But this would imply that $z$ and $\overline{z}$ are linearly dependent (in $\mathbb{C}^d$), which cannot be, since they are associated to the respective distinct eigenvalues $\lambda$ and $\overline{\lambda}$. Therefore, the real and imaginary parts of $z$ are linearly dependent.\\
In the end, observing that:
\begin{align*}
\Lambda_0 u_1 = \Lambda_0 \left( \frac{z + \overline{z}}{2}\right) = \frac{1}{2} \left( \lambda z + \overline{\lambda} \overline{z} \right) = \mathfrak{Re} (\lambda z) = - \text{sgn}(\lambda) \vert \lambda \vert u_2 \neq 0,
\end{align*}
and
\begin{align*}
\Lambda_0 u_2 = \overline{\Lambda}_0 \left( \frac{z - \overline{z}}{2i}\right) = \frac{1}{2i} \left( \lambda z - \overline{\lambda} \overline{z} \right) = \mathfrak{Im} (\lambda z) = \text{sgn}(\lambda) \vert \lambda \vert u_1 \neq 0.
\end{align*}
We can therefore conclude that $\text{Span}(u_1,u_2)$ is of dimension $2$, and that this space is in direct sum with $\text{Ker} \Lambda_0$. In other words, the kernel of $\Lambda_0$ has a dimension that is at most $d-2$. Therefore, in the case $\Lambda_0 \neq 0$, \eqref{EQUATChptrLgTBhBounceBackCond3} can never hold, regardless the dimension $d$, or the geometry of the boundary $\partial \Omega$, because no hypersurface can be contained in a vector space of codimension $2$.\\
This implies, according to \eqref{EQUATChptrLgTBhBounceBackCond2}, that $w_2 = 0$, and so Theorem \ref{THEORSect3BounceBack_B_C_} is proved.
\end{proof}

\begin{remar}
This is equation \eqref{EQUATChptrLgTBhBounceBackCond3} that is used by Desvillettes in \cite{Desv990} to deduce that if $\Omega$ is bounded, then we obtain a contradiction if $\Lambda_0 \neq 0$. Our addition here is the discussion concerning the dimension of the kernel of $\Lambda_0$, which enables to obtain the result of Theorem \ref{THEORSect3BounceBack_B_C_}, without using any boundedness assumption on $\Omega$.
\end{remar}

\section{Local Maxwellians and boundary condition II: the specular reflection case}

Let us now turn to the case of the specular reflection, which is much more relevant regarding the physical motivation. In this case, the particles are assumed to be reflected against the boundary of the domain, exactly as a billiard ball would do: during the collision, the velocity of the particle is reflected, that is, is obtained as the orthogonal symmetry, with respect to the tangent hyperplane to the obstacle, at the point of bounce.
As for the bounce-back boundary condition, this assumption on the behaviour of the particles is translated into an equality on the density function $f$.

\begin{defin}[Specular reflection boundary condition (SRBC)]
Let $\Omega$ be a regular, open set in $\mathbb{R}^d$, and let $f : \overline{\Omega} \rightarrow \mathbb{R}$ be any function. $f$ is said to satisfy the \emph{specular reflection boundary condition} (abbreviated as \emph{SRBC}) on the boundary $\partial \Omega$ of $\Omega$ if:
\begin{align}
\label{EQUATChptrLgTBhSpecuRefle_B_C_}
\forall t \in [0,T], x \in \partial \Omega, v \in \mathbb{R}^d, f(t,x,v) = f(t,x,v'),
\end{align}
where $v'$ is defined as:
\begin{align}
\label{EQUATChptrLgTBhSpecuRefleReflx}
v' = v - 2 \left( v \cdot n(x) \right) n(x),
\end{align}
with $n(x)$ a unitary normal vector to the boundary $\partial \Omega$ of the domain, at $x \in \partial \Omega$.
\end{defin}

\begin{remar}
Note that we did not specify completely the normal vector $n(x)$ we used in \eqref{EQUATChptrLgTBhSpecuRefleReflx}. The reason is that changing $n(x)$ into $-n(x)$ leaves the expression \eqref{EQUATChptrLgTBhSpecuRefleReflx} of $v'$ unchanged.
\end{remar}
\noindent
We will now see that, contrary to the bounce-back boundary condition, there can be much more local Maxwellians solving the free transport equation with these new boundary condtions. More precisely, we will see that, the more the domain is ``symmetric'', the more we can find local Maxwellians verifying the free transport with SRBC.\\
We will proceed first when the dimension $d$ is $2$, and then when $d=3$. However, the starting point, namely, exploiting the boundary condition \eqref{EQUATChptrLgTBhSpecuRefle_B_C_}, is independent from the dimension, and provides the equation:
\begin{align}
\left\vert v - u(t,x) \right\vert^2 = \left\vert v' - u(t,x) \right\vert^2
\end{align}
for all $t \in [0,T]$ and $x \in \partial \Omega$ (where $v'$ is defined in \eqref{EQUATChptrLgTBhSpecuRefleReflx}),
as a consequence of the expression \eqref{EQUATSect2ExpreMaxweLocal}. The last equation can be rewritten as:
\begin{align}
\left( v \cdot n(x) \right) n(x) \cdot u(t,x) = 0,
\end{align}
for all $t \in [0,T]$, $x \in \partial \Omega$ and $v \in \mathbb{R}^d$. Since $v$ can be chosen freely in $\mathbb{R}^d$, we have:
\begin{align*}
n(x) \cdot u(t,x) = 0 \hspace{5mm} \forall t \in [0,T], x \in \partial \Omega,
\end{align*}
which we rewrite again, taking advantage of the explicit expression of $u(t,x)$ we derived in \eqref{EQUATSS2.3DeuxiExpre__u__} page \pageref{EQUATSS2.3DeuxiExpre__u__}, \eqref{EQUATSS2.4ExpreCompl__C__} page \pageref{EQUATSS2.4ExpreCompl__C__}, and \eqref{EQUATSS2.4ExpreCompl__a__} page \pageref{EQUATSS2.4ExpreCompl__a__}:
\begin{align}
\label{EQUATChptrLgTBhCond_u(tx)}
n(x) \cdot \left[ \Lambda_0 x + \left(\alpha t + \beta/2\right) x + t w_1 + w_2 \right] = 0 \hspace{5mm} \forall t \in [0,T], x \in \partial \Omega.
\end{align}
Equation \eqref{EQUATChptrLgTBhCond_u(tx)} will then be the central object of our study in the present section. $\alpha,\beta \in \mathbb{R}$, $w_1,w_2 \in \mathbb{R}^d$ and $\Lambda_0$ skew-symmetric, all have to be determined, depending on the shape of the domain, and the dimension $d$.\\
\newline
\noindent
In the case of the specular reflection case, the condition \eqref{EQUATChptrLgTBhCond_u(tx)} we obtained on the boundary $\partial \Omega$ is less easy to study than \eqref{EQUATSect3CondiGener__u__} (obtained in the BBBC case). Indeed, this time the information provided by \eqref{EQUATChptrLgTBhCond_u(tx)} is only along the direction of the normal vector $n(x)$, which depends itself on $x$.
In order to establish the proof of Theorems \ref{THEORChptrLgTBhSpecuRefle_B_C_} and \ref{THEORChptrLgTBhSpecuRefleBCd=3}, we will make use of the following results, that describe solutions of some particular linear ODEs in general dimensions.\\
These results will be applied to describe which type of curves the boundary $\partial\Omega$ of an open set $\Omega$ can contain, if one knows particular conditions of the form of \eqref{EQUATChptrLgTBhCond_u(tx)} holding on the normal vector $n(x)$ to the boundary $\partial\Omega$. The idea to consider curves drawn on the boundary $\partial\Omega$ can already be found in \cite{Desv990}.

\begin{lemma}
\label{LEMMEChptrLgTBhSRODEAlpha}
Let $\alpha \in \mathbb{R}$ such that $\alpha \neq 0$, and let $w_1 \in \mathbb{R}^d$ (with $d = 2$ or $d = 3$).\\
Then the solution $t \mapsto x(t)$ of the the following Cauchy problem:
\begin{align}
\label{EQUATChptrLgTBhSRODEAlpha}
\left\{
\begin{array}{rcl}
\frac{\dd}{\dd t} x(t) &=& \alpha x(t) + w_1,\\
x(0) &=& x_0 \in \mathbb{R}^d,
\end{array}
\right.
\end{align}
is globally well-defined and is given by the following expression:
\begin{align}
x(t) = \frac{1}{\alpha} \left( \alpha x_0 + w_1 \right) e^{\alpha t} - \frac{w_1}{\alpha} \hspace{5mm} \forall t \in \mathbb{R}.
\end{align}
\end{lemma}

\begin{lemma}
\label{LEMMEChptrLgTBhSRODELambd}
Let $\Lambda_0 \in \mathcal{M}_d(\mathbb{R})$ be a non-zero skew-symmetric matrix, let $\beta \in \mathbb{R}$ be a real number, and let $w_2 \in \mathbb{R}^d$ (with $d = 2$ or $d = 3$).\\
Then the solution $t \mapsto x(t)$ of the the following Cauchy problem:
\begin{align}
\label{EQUATChptrLgTBhSRODELambd}
\left\{
\begin{array}{rcl}
\frac{\dd}{\dd t} x(t) &=& \Lambda_0 x(t) + \frac{\beta}{2} x(t) + w_2,\\
x(0) &=& x_0 \in \mathbb{R}^d,
\end{array}
\right.
\end{align}
is globally well-defined. In addition, the solution $t \mapsto x(t)$ can be described explicitely as follows.
\begin{itemize}
\item If $\beta \neq 0$ or if the dimension $d$ is equal to $2$, the matrix $\Lambda_0 + \frac{\beta}{2}I_d$ is invertible, and the solution $t \mapsto x(t)$ is given by the expression:
\begin{align}
\label{EQUATExpreSolut_EDO_LambdBetn0}
x(t) = e^{\frac{\beta}{2} t} e^{t \Lambda_0} \left( x(0) + y \right) - y,
\end{align}
where $y$ is the preimage of the vector $w_2$ by the matrix $\Lambda_0 + \frac{\beta}{2}I_d$.\\
In such a case, the graph of $t \mapsto x(t)$ is a logarithmic spiral inscribed in a plane that contains the point $-y$ if $\beta \neq 0$, and a circle inscribed in a plane if $\beta = 0$.
\item If $d = 3$ and $\beta = 0$, the solution $t \mapsto x(t)$ is given by the expression:
\begin{align}
\label{EQUATExpreSolut_EDO_LambdBet=0}
x(t) = e^{t \Lambda_0} \left(x(0) + y\right) - y + \lambda u t,
\end{align}
where $u = z/\vert z \vert$, with $z$ being the unique non-zero vector such that $\Lambda_0 x = z \wedge x$ for all $x \in \mathbb{R}^3$, and $\lambda \in \mathbb{R}$, $y \in \mathbb{R}^3$ are such that $w = \Lambda_0 y$, with:
\begin{align}
w_2 = \lambda u + w = \lambda \frac{z}{\vert z \vert} + w, \hspace{3mm} \text{with} \hspace{3mm} w \perp u.
\end{align}
In such a case, the graph of $t \mapsto x(t)$ is an helix of axis $\text{Span}(u) = \text{Span}(z/\vert z \vert)$ if $\lambda \neq 0$, and a circle inscribed in a plane if $\lambda = 0$.
\end{itemize}
\end{lemma}

\begin{proof}
In the case when $\beta \neq 0$, the matrix $\Lambda_0 + \frac{\beta}{2}I_d$ is invertible. In the case when $\beta = 0$ and $d =2$, then this matrix can be written explicitely as $\begin{pmatrix} 0 & -a \\ a & 0 \end{pmatrix}$, with $a \neq 0$ (because $\Lambda_0$ was assumed to be not zero), which is also invertible. In both cases, we can define the vector $y$, preimage of $w_2$ by the matrix $\Lambda_0 + \frac{\beta}{2} I_d$. Therefore, posing:
\begin{align}
\varphi(t) = (x(t) + y) e^{-\frac{\beta}{2}t},
\end{align}
we find:
\begin{align}
\frac{\dd}{\dd t}\varphi(t) &= \left( \Lambda_0 x(t) + \frac{\beta}{2} x(t) + w_2\right)e^{-\frac{\beta}{2}t} - \frac{\beta}{2}\left( x(t) + y \right) e^{-\frac{\beta}{2}t} \nonumber\\
&= \left( \Lambda_0 x(t) + \Lambda_0 y \right) e^{-\frac{\beta}{2}t} = \Lambda_0 \varphi(t),
\end{align}
hence
\begin{align}
\varphi(t) = e^{t \Lambda_0} \varphi(0) = e^{t \Lambda_0} \left( x(0) + y \right).
\end{align}
We deduce then the expression \eqref{EQUATExpreSolut_EDO_LambdBetn0}.\\
In the second case, since $\Lambda_0$ was assumed to be non-zero and the dimension $d$ is assumed to be equal to $3$, there exists a non-zero vector $z$ such that the action of $\Lambda_0$ can be represented by the cross product with the vector $z$. Denoting $z/\vert z \vert$ by $u$, we decompose:
\begin{align}
w_2 = \lambda u + w,
\end{align}
for some $\lambda \in \mathbb{R}$ and $w \in \mathbb{R}^3$, such that $w \perp u$. Such a decomposition satisfies that $w$ belongs to the image of $\Lambda_0$, let us denote by $y$ one of its preimages (considering for instance $y = \left((-z)/\vert z \vert^2\right)\wedge w$, we find $\Lambda_0 y = w$).\\
This time, we pose:
\begin{align}
\varphi(t) = x(t) + y - \lambda u t,
\end{align}
and in this case we find:
\begin{align}
\frac{\dd}{\dd t} \varphi(t) &= \Lambda_0 x(t) + w_2 - \lambda u \nonumber\\
&= \Lambda_0 x(t) + \Lambda_0 y = \Lambda_0\left( x(t)+y-\lambda u t\right)
\end{align}
(because $u$ is colinear to $z$, so that $\Lambda_0 u = 0$), that is:
\begin{align}
\frac{\dd}{\dd t} \varphi(t) = \Lambda_0 \varphi(t).
\end{align}
We deduce therefore the expression \eqref{EQUATExpreSolut_EDO_LambdBet=0} for the solution $x(t)$.\\
Finally, the geometric description of the graph of the solution $t \mapsto x(t)$ is a consequence of the explicit computation of the exponential of the matrix $t \Lambda_0$. More precisely, in the $2$-dimensional case $d=2$, since we have:
\begin{align*}
\Lambda_0^2 = -a^2 I_2,
\end{align*}
we find:
\begin{align}
\Lambda_0^{2k} = (-a^2)^k I_2 \hspace{5mm} \text{and} \hspace{5mm} \Lambda_0^{2k+1} = (-a^2)^k \Lambda_0 \hspace{5mm} \forall k \in \mathbb{N}.
\end{align}
Then, we obtain:
\begin{align*}
\exp\left(t \Lambda_0\right) &= \sum_{k=0}^{+\infty} \frac{t^{2k}}{(2k)!} (-a^2)^k I_2 + \sum_{k=0}^{+\infty} \frac{t^{2k+1}}{(2k+1)!}(-a^2)^k \Lambda_0 \\
&= \begin{pmatrix} \sum_{k=0}^{+\infty} (-1)^k\frac{t^{2k}}{(2k)!}a^{2k} & \sum_{k=0}^{+\infty} (-1)^k\frac{t^{2k+1}}{(2k+1)!}a^{2k+1} \\ -\sum_{k=0}^{+\infty} (-1)^k\frac{t^{2k+1}}{(2k+1)!}a^{2k+1} & \sum_{k=0}^{+\infty} (-1)^k\frac{t^{2k}}{(2k)!}a^{2k} \end{pmatrix} \\
&= \begin{pmatrix} \cos(at) & \sin(at) \\ -\sin(at) & \cos(at) \end{pmatrix},
\end{align*}
that is, the matrix $e^{t\Lambda_0}$ is the matrix of a rotation.\\
The $3$-dimensional case $d=3$ is exactly similar, since there is always a choice of coordinate such that $z$ is colinear to the last vector of the canonical basis, providing in such a case that:
\begin{align}
\Lambda_0 = \begin{pmatrix} 0 & -a & 0 \\ a & 0 & 0 \\ 0 & 0 & 0 \end{pmatrix},
\end{align}
leading to the same conclusion.
\end{proof}

\subsection{Local Maxwellians and specular reflection, in dimension $d = 2$}

Before stating the main theorem of this section, we introduce the notion of rotational symmetry, that we will use to write the main result of this section.

\begin{defin}[Rotational symmetry, planar version]
\label{DEFINRotatSymme_Plan}
Let $\mathcal{C}$ be a $\mathcal{C}^1$ curve of the Euclidean plane $\mathbb{R}^2$. $\mathcal{C}$ is said to present a \emph{rotational symmetry} if there exists a point $x_0 \in \mathbb{R}^2$ of the plane such that, for any point $x \in \mathcal{C}$ and any affine rotation $\mathcal{R}_{x_0}$ around $x$, the image $\mathcal{R}_{x_0}(x)$ of $x$ by the rotation $\mathcal{R}_{x_0}$ belongs also to the curve $\mathcal{C}$.
\end{defin}

\begin{remar}
If the connected curve $\mathcal{C}$ presents a rotational symmetry around a certain point $x_0$, then $\mathcal{C}$ is a circle of center $x_0$. Therefore, if $\Omega$ is a connected regular open set of the plane, such that its (non empty) boundary presents a rotational symmetry, then $\Omega$ is either a disk, the complement of a disk, of an annulus contained between two concentric circles.
\end{remar}

\noindent
We can now turn to the characterization of the local Maxwellians solving the Boltzmann equation in the $2$-dimensional case, in the case of the specular reflection boundary condition.

\begin{theor}[Characterization of the local Maxwellians solving the free transport equation in a domain with specular reflection boundary condition (SRBC), case $d=2$]
\label{THEORChptrLgTBhSpecuRefle_B_C_}
Let the dimension $d$ be equal to $2$. Let $T$ be a strictly positive number (possibly $+\infty$), let $\Omega$ be a regular, connected open set of $\mathbb{R}^2$ and let $m$ be a local Maxwellian  that is, a function of the form \eqref{EQUATSS2.1ExpreMaxweLocGn}, which solves the transport equation \eqref{EQUATSect2Transport_Libre} on $[0,T]\times\Omega\times\mathbb{R}^2$, with the specular reflection boundary condition (SRBC), that is, such that
\begin{align}
\forall t \in [0,T], x \in \partial \Omega, v \in \mathbb{R}^d,\hspace{3mm} m(t,x,v) = m(t,x,v'),
\end{align}
where
\begin{align}
v' = v - 2\left( v\cdot n(x) \right) n(x)
\end{align}
and with $n(x)$ a unitary normal vector to $\partial\Omega$ at $x$.\\
Then,
\begin{itemize}
\item if $\Omega$ is a half-plane, with a boundary of the form $\partial\Omega = \{x_0 n + \lambda \tau\ /\ \lambda \in \mathbb{R}\}$, with $x_0 \in \mathbb{R}$, $n,\tau \in \mathbb{S}^1$ such that $n\perp \tau$ (that is, $x_0 n$ is the closest point of boundary $\partial\Omega$ to the origin, and $\partial\Omega$ is orientated by the unit vector $\tau$), there exist six real numbers $r_0>0$, $\alpha>0$, $\beta$, $\gamma>0$, $\ell_1$ ,$\ell_2 \in \mathbb{R}$ such that:
\begin{align}
\label{EQUATLocalMaxweSpecu_d=2_1Plan}
m(t,x,v) &= r_0 \exp\big( - \alpha (x-tv-2x_0 n)\cdot(x-tv) + \beta (x-tv-x_0n)\cdot v - \gamma \vert v \vert^2 \nonumber\\
&\hspace{80mm}- 2\ell_1 \tau \cdot(x-tv) + 2 \ell_2 \tau \cdot v \big) \nonumber\\
&\hspace{90mm}\forall t \in [0,T], x \in \partial \Omega, v \in \mathbb{R}^d,
\end{align}
\item if $\Omega$ is a slab, with a boundary of the form $\partial\Omega = \{x_1 n + \lambda \tau\ /\ \lambda \in \mathbb{R}\} \cup \{ x_2n + \mu \tau\ /\ \mu \in \mathbb{R}\}$, with $x_0 \in \mathbb{R}$, $n,\tau \in \mathbb{S}^1$ such that $n\perp \tau$, there exist four real numbers $r_0>0$, $\gamma>0$, $\ell_1, \ell_2 \in \mathbb{R}$ such that:
\begin{align}
\label{EQUATLocalMaxweSpecu_d=2_2Plan}
m(t,x,v) &= r_0 \exp\left( - \gamma \vert v \vert^2 - 2 \ell_1\tau\cdot(x-tv) + 2\ell_2\tau\cdot v \right) \hspace{5mm} \forall t \in [0,T], x \in \partial \Omega, v \in \mathbb{R}^d,
\end{align}
\item if $\Omega$ is a disk centered on $y \in \mathbb{R}^2$, or the complement of a disk centered on $y \in \mathbb{R}^2$, or an annulus contained between two concentric circles centered on $y \in \mathbb{R}^2$, there exist two strictly positive real numbers $r_0>0$ and $\gamma > 0$, and a skew-symmetric matrix $\Lambda_0$ such that:
\begin{align}
\label{EQUATChptrLgTBhMaxwlSpecu_Disk}
m(t,x,v) = r_0 \exp\left( -\gamma \vert v \vert^2 + 2 \left[ \Lambda_0 (x-y) \right] \cdot v \right) \hspace{5mm} \forall t \in [0,T], x \in \partial \Omega, v \in \mathbb{R}^d.
\end{align}
\item finally, if $\Omega$ is \emph{not} a half-plane, nor a slab, nor a disk, the complement of a disk, nor an annulus contained between two concentric circles, there exist two strictly positive real numbers $r_0>0$ and $\gamma>0$ such that:
\begin{align}
\label{EQUATChptrLgTBhMaxwlSpecuNDisk}
m(t,x,v) = r_0 \exp\left( -\gamma \vert v \vert^2 \right) \hspace{5mm} \forall t \in [0,T], x \in \partial \Omega, v \in \mathbb{R}^d,
\end{align}

\end{itemize}
\end{theor}

\begin{remar}
It is interesting to observe that, in the case when the domain $\Omega$ is not bounded, we find for the local Maxwellians the expressions \eqref{EQUATLocalMaxweSpecu_d=2_1Plan} and \eqref{EQUATLocalMaxweSpecu_d=2_2Plan}, that are not steady states of the Boltzmann equation.
\end{remar}

\begin{proof}[Proof of Theorem \ref{THEORChptrLgTBhSpecuRefle_B_C_}]
By Theorem \ref{THEORSS2.1ExpreMaxweLocGn}, a local Maxwellian that solves the free transport equation in an open domain is necessarily of the form \eqref{EQUATSS2.1ExpreMaxweLocGn}. As for the bounce-back case, we start with exploiting what implies the boundary condition on the vector field $u(t,x)$, starting from the expression \eqref{EQUATChptrLgTBhCond_u(tx)}.\\
At $x$ fixed, the left hand side of \eqref{EQUATChptrLgTBhCond_u(tx)} is polynomial expression in $t$, which is zero for the non trivial interval $[0,T] \owns t$, therefore all the coefficients of the expression have to be zero, so that we deduce the following system: $\forall x \in \partial \Omega$ we have

\begin{subequations}
\label{EQUATChptrLgTBhSyst_u(tx)}
\begin{empheq}[left=\empheqlbrace]{align}
\left[ \alpha x + w_1 \right] \cdot n(x) = 0, \label{EQUATChptrLgTBhSystu(tx)1}\\
\left[ \Lambda_0 \cdot x + \displaystyle{\frac{\beta}{2}} x + w_2 \right] \cdot n(x) = 0, \label{EQUATChptrLgTBhSystu(tx)2}
\end{empheq}
\end{subequations}

\noindent
Therefore, we start with \eqref{EQUATChptrLgTBhSystu(tx)1}, and we consider the Cauchy problem \eqref{EQUATChptrLgTBhSRODEAlpha}, such that the initial datum $x_0$ belongs to the boundary $\partial\Omega$ of the domain $\Omega$. In such a case, the curve $x(t)$ always lies on the boundary $\partial \Omega$ of the domain. Indeed, the following result holds: if a parametric curve $I \subset \mathbb{R} \rightarrow \mathbb{R}^d, t \mapsto x(t)$ is such that $\frac{\dd}{\dd t} x(t) \cdot \nabla g\left( x(t) \right) = 0$ for a given function $\mathbb{R}^d \rightarrow \mathbb{R}, x \mapsto g(x)$ and for every $t \in I$, then the curve $x(t)$ lies in the surface $S$ defined implicitely as $S=\{x \in \mathbb{R}^d \ /\ g(x) = 0\}$.\\
Let us assume in addition that $\alpha \neq 0$. In such a case, we can consider the initial datum $x_0 \in \partial \Omega$ such that $x_0 \neq -w_1/\alpha$ (because the boundary $\partial \Omega$ is not empty, and cannot be reduced to a single point). We deduce by the formula \eqref{EQUATChptrLgTBhSRODEAlpha} of Lemma \ref{LEMMEChptrLgTBhSRODEAlpha} that the boundary $\partial \Omega$ contains the half-line with initial point $-w_1/\alpha$, which is passing through the point $x_0$. Considering that $\partial \Omega$ is a regular curve (and considering in particular its tangent space at the point $-w_1/\alpha$), we deduce that $\partial \Omega$ has to be reduced to a straight line through $-w_1/\alpha$. Therefore, $\Omega$ has to be a half plane, such that $-w_1/\alpha$ belongs to its boundary.\\
If now $\alpha = 0$ and $w_1 \neq 0$, \eqref{EQUATChptrLgTBhSystu(tx)1} implies that $\partial \Omega$ has to be the union of straight lines parallel to $w_1$. Since in addition $\Omega$ was assumed to be connected, $\partial \Omega$ is either a single straight line, or the union of two parallel lines, and so $\Omega$ is either a half plane, or a stripe.\\
\newline
We turn now to the second equation \eqref{EQUATChptrLgTBhSystu(tx)2}. We consider a solution of the Cauchy problem \eqref{EQUATChptrLgTBhSRODELambd}, with initial datum $x_0 \in \partial\Omega$, so that the whole curve described by the solution belongs to the boundary $\partial\Omega$.\\
In a first time, let us assume that $\Lambda_0 \neq 0$. According to Lemma \ref{LEMMEChptrLgTBhSRODELambd}, the curve of the solution is either a logarithmic spiral, or a circle. In the first case, one would reach a contradiction. Indeed, $\partial\Omega$ being the boundary of the domain $\Omega$, $\partial\Omega$ is in particular closed by assumption, so that the center of the spiral has to belong to $\partial\Omega$. However, at this point the curve does not admit a tangent line, which would violate the assumption that $\Omega$ is a regular open set. Therefore, $\partial\Omega$ cannot be a logarithmic spiral, and so if $\Lambda_0 \neq 0$, then $\beta = 0$, and the boundary $\partial\Omega$ is the union of concentric circles, centered on $-y$. Since $\Omega$ was assumed to be connected, $\partial\Omega$ is either reduced to a single circle, or to two concentric circles, and $\Omega$ is either a disk, or the complement of a disk, or an annulus contained between two concentric circles.\\
If now $\Lambda_0 = 0$, then the same conclusions hold for the boundary as the conclusions obtained from the first equation \eqref{EQUATChptrLgTBhSystu(tx)1}.\\
\newline
In summary, we obtained the following results.
\begin{itemize}
\item if $\alpha \neq 0$, then $\Omega$ is a half-plane, with $-w_1/\alpha \in \partial\Omega$,
\item if $\alpha = 0$ and $w_1 \neq 0$, then $\Omega$ is either a half-plane or a stripe, with a boundary $\partial\Omega$ constituted with one or two straight lines parallel to $w_1$,
\item if $\Lambda_0 \neq 0$, then $\beta = 0$ and $\partial\Omega$ is either a disk, or the complement of a disk, or an annulus contained between two concentric disks, all centered on $-y = (\Lambda_0)^{-1}(w_2)$,
\item if $\Lambda_0 = 0$ and $\beta \neq 0$, then $\Omega$ is a half-plane, with $-2w_2/\beta \in \partial\Omega$,
\item if $\Lambda_0 = 0$, $\beta = 0$ and $w_2 \neq 0$, then $\Omega$ is either a half-plane or a stripe, with a boundary $\partial\Omega$ constituted with one or two straight lines parallel to $w_2$.
\end{itemize}
The result of Theorem \ref{THEORChptrLgTBhSpecuRefle_B_C_} follows.
\end{proof}

\noindent
It is usual in the literature (\cite{Desv990}, \cite{ReVi008}) to consider only bounded domains $\Omega$, that are in addition simply connected. This particular case forbids to consider domains such as stripes or half-planes, and the case of $\Omega$ being an annulus is also excluded. We can then recover the following result, classical in the literature, as a direct consequence of Theorem \ref{THEORChptrLgTBhSpecuRefle_B_C_}.

\begin{corol}[Characterization of the local Maxwellians solving the free transport equation in a bounded domain with specular reflection boundary condition (SRBC), case $d=2$]
\label{COROLChptrLgTBhSpecuRefle_B_C_}
Let the dimension $d$ be equal to $2$. Let $T$ be a strictly positive number (possibly $+\infty$), let $\Omega$ be a regular open set of $\mathbb{R}^2$, bounded and simply connected and let $m$ be a local Maxwellian  that is, a function of the form \eqref{EQUATSS2.1ExpreMaxweLocGn}, which solves the transport equation \eqref{EQUATSect2Transport_Libre} on $[0,T]\times\Omega\times\mathbb{R}^2$, with the specular reflection boundary condition (SRBC), that is, such that
\begin{align}
\forall t \in [0,T], x \in \partial \Omega, v \in \mathbb{R}^d,\hspace{3mm} m(t,x,v) = m(t,x,v'),
\end{align}
where
\begin{align}
v' = v - 2\left( v\cdot n(x) \right) n(x)
\end{align}
and with $n(x)$ a unitary normal vector to $\partial\Omega$ at $x$.\\
Then,
\begin{itemize}
\item if $\Omega$ is \emph{not} a disk, there exist two strictly positive real numbers $r_0>0$ and $\gamma>0$ such that:
\begin{align}
\label{EQUATChptrLgTBhMaxwlSpecuNDis2}
m(t,x,v) = r_0 \exp\left( -\gamma \vert v \vert^2 \right) \hspace{5mm} \forall t \in [0,T], x \in \partial \Omega, v \in \mathbb{R}^d,
\end{align}
\item if $\Omega$ is a disk centered on $y \in \mathbb{R}^2$, there exist two strictly positive real numbers $r_0>0$ and $\gamma > 0$, and a skew-symmetric matrix $\Lambda_0$ such that:
\begin{align}
\label{EQUATChptrLgTBhMaxwlSpecu_Dis2}
m(t,x,v) = r_0 \exp\left( -\gamma \vert v \vert^2 + 2 \left[ \Lambda_0 (x-y) \right] \cdot v \right) \hspace{5mm} \forall t \in [0,T], x \in \partial \Omega, v \in \mathbb{R}^d.
\end{align}
\end{itemize}
\end{corol}

\begin{remar}
Corollary \ref{COROLChptrLgTBhSpecuRefle_B_C_} states that if the boundary $\partial\Omega$ of the domain does not present a rotational symmetry in the sense of Definition \ref{DEFINRotatSymme_Plan}, then, there exists a single local Maxwellian solving the free transport equation with specular reflection boundary condition (for mass and kinetic energy fixed), and such a Maxwellian is actually global.\\
On the other hand, if the domain is rotationnaly symmetric, then there exist other local Maxwellians solving the free transport equation with SRBC. Such Maxwellians do not depend on time, but the coefficients do depend on the position variable $x$.\\
Let us also remark that the additional degree of freedom obtained in the symmetric case, that is, the possibility to choose any skew-symmetric matrix $\Lambda_0$ to define $m$ as in \eqref{EQUATChptrLgTBhMaxwlSpecu_Disk}, corresponds actually to a single additional dimension of freedom. Indeed, a skew-symmetric matrix in dimension $2$ necessarily writes:
\begin{align*}
\Lambda_0 = \begin{pmatrix} 0 & -a \\ a & 0 \end{pmatrix},
\end{align*}
where $a$ is an arbitrary real number.
\end{remar}

\subsection{Local Maxwellians and specular reflection, in dimension $d = 3$}

We conclude this study with the second physically relevant case (and maybe the most relevant of the two): the $3$-dimensional case.\\
Before stating the result, we start with introducing definitions concerning particular symmetries that we will use in the following.

\begin{defin}[Rotational symmetry, spatial version]
\label{DEFINRotatSymme_Spat}
Let $\mathcal{S}$ be a $\mathcal{C}^1$ surface of the Euclidean space $\mathbb{R}^3$. $\mathcal{S}$ is said to present a \emph{rotational symmetry} if there exists an affine straight line $\Delta \subset \mathbb{R}^3$ of the space such that, for any point $x \in \mathcal{S}$ and any affine rotation $\mathcal{R}_{x_0}$ around $\Delta$, the image $\mathcal{R}_{\Delta}(x)$ of $x$ by the rotation $\mathcal{R}_{\Delta}$ belongs also to the surface $\mathcal{S}$.
\end{defin}

\begin{remar}
A surface $\mathcal{S}$ that presents a rotational symmetry is the union of circles that are all inscribed in planes orthogonal to the axis of symmetry $\Delta$.
\end{remar}

\begin{defin}[Helical symmetry]
\label{DEFINRotatSymme_Spat}
Let $\mathcal{S}$ be a $\mathcal{C}^1$ surface of the Euclidean space $\mathbb{R}^3$. $\mathcal{S}$ is said to present an \emph{helical symmetry} if there exists an affine straight line $\Delta \subset \mathbb{R}^3$ of the space, orientated by a unit vector $v \in \mathbb{S}^2$, and a real number $\lambda \in \mathbb{R}$ such that, for any point $x \in \mathcal{S}$ and any affine rotation $\mathcal{R}_{\Delta,\theta}$ around $\Delta$ of angle $\theta$, the image $\mathcal{R}_{\Delta,\theta}(x) + \lambda \theta u$ of $x$ by the screw motion $\mathcal{R}_{\Delta,\theta} + \lambda \theta u$ belongs also to the surface $\mathcal{S}$.\\
The number $2\pi\lambda$ is called the \emph{shift} of the helical symmetry.
\end{defin}

\begin{remar}
In the case when a surface $\mathcal{S}$ presents an helical symmetry, the surface is the union of helices, that all have the same axis and the same shift (that is, the same distance between two consecutive coils).\\
In the case when $\lambda = 0$, the notion of helical symmetry coincides with the notion of rotational symmetry.
\end{remar}

\noindent
We introduce a precise nomenclature for remarkable surfaces that will appear in the main result of this section.

\begin{defin}[Generalized cylinder]
A surface $\mathcal{S}$ of the Euclidean space $\mathbb{R}^3$ is said to be a \emph{generalized cylinder} if there exists a family $\left(d_s\right)_{s \in I}$ of straight lines $d_s$, all parallel, such that $\mathcal{S}$ is the union of all of the lines $d_s$, for $s \in I$.\\
We call the \emph{direction} of the generalized cylinder $\mathcal{S}$ any unit vector that orientates any of the lines $d_s$.
\end{defin}

\begin{remar}
A generalized cylinder is in particular a ruled surface.
\end{remar}

\begin{defin}[Cylinder of revolution]
A surface $\mathcal{S}$ of the Euclidean space $\mathbb{R}^3$ is said to be a \emph{cylinder of revolution} if $\mathcal{S}$ is the union of a family of parallel straight lines obtained from all the possible rotations of a given line $d$ around an axis $\Delta$ parallel to $d$.
\end{defin}

\begin{remar}
In particular, a cylinder of revolution is a generalized cylinder such that all the lines $d_s$ intersect a given circle inscribed in a plane orthogonal to the axis $\Delta$.
\end{remar}

\noindent
Finally, we will make use of the two following results concerning surfaces that are generalized cylinders or presenting an helical symmetry.\\
The first result describes the intersection between the class of the surfaces presenting an helical symmetry, with the class of the generalized cylinders.

\begin{lemma}
\label{LEMMECylinGenerSymetHelic}
Let $\mathcal{S}$ be a $\mathcal{C}^1$ connected surface of the Euclidean space $\mathbb{R}^3$, that is a generalized cylinder. Then, if in addition $\mathcal{S}$ presents an helical symmetry:
\begin{itemize}
\item if the shift $\lambda$ of the helical symmetry is zero, $\mathcal{S}$ is either a plane, or a cylinder of revolution,
\item if the shift $\lambda$ of the helical symmetry is non-zero, $\mathcal{S}$ is a cylinder of revolution.
\end{itemize}
\end{lemma}

\begin{proof}
By assumption, $\mathcal{S}$ presents an helical symmetry (say, of axis $\Delta$), and contains at least one straight line $d$.\\
Our definition of helical symmetry covers also the case of a rotational symmetry. Let us start with the case, that is, we assume in a first time that the shift of the helical symmetry is zero.\\
Let us assume first that $d$ intersects the axis of symmetry $\Delta$, in which case these two lines are either orthogonal or parallel, or none of these two cases. In the two first cases, $\mathcal{S}$ is a plane, or a cylinder of revolution. In the latter case, that is, when $d$ and $\Delta$ intersect with an angle different from $k\pi/2$, $k = 0$ or $1$, the rotational motion of $d$ around $\Delta$ generates a cone, which is not a $\mathcal{C}^1$ surface in the neighbourhood of the intersection between $d$ and $\Delta$.\\
If now we assume that $d$ does not intersect the axis $\Delta$, if $d$ and $\Delta$ are parallel or orthognal, the surface generated by the rotational motion of $d$ around the axis $\Delta$ is either a cylinder of revolution, or is contained in a plane. If now $d$ and $\Delta$ are neither parallel nor orthogonal, the surface generated by the rotational motion of $d$ is a one-sheeted hyperboloid of revolution, which is not a generalized cylinder. This case is therefore excluded.\\
The case of an helical symmetry with a zero shift is then completely addressed.\\
\newline
Let us now assume that the helical symmetry has a non-zero shift. We will make use of the following result, that the reader may find for instance in \cite{Stru967}, to eliminate most of the cases: if the line $d$ is not colinear to $\Delta$, nor orthogonal to $\Delta$, then the helicoid generated by the screw motion of $d$ around the axis $\Delta$ presents a self-intersection. Such a surface is called an \emph{helicoid with directrix cone}, or an \emph{helicoid of oblique type}. Therefore $\mathcal{S}$, which would contain such an oblique helicoid, cannot be a $\mathcal{C}^1$ surface. This case is then excluded.\\
If $d$ and $\Delta$ are parallel, the helical motion of $d$ around $\Delta$ generates a cylinder of revolution.\\
The only remaining case is then when $d$ and $\Delta$ are orthogonal. The helical motion of $d$ around $\Delta$ generates a \emph{right helicoid} (\emph{closed} or \emph{open}, depending if $d$ intersects $\Delta$ or not). Such an helicoid is by definition a ruled surface, however we will now show that it cannot be a generalized cylinder. If a right helicoid would be a generalized cylinder, its direction could not be parallel to $\Delta$, because it would contain a plane such that its normal is orthogonal to $\Delta$, which is not the case. For the same reason, its direction cannot be orthogonal to $\Delta$, because the helicoid would contain a plane that is orthogonal to $\Delta$. Finally, if its direction is not parallel nor orthogonal to $\Delta$, then the helicoid would be of oblique type, which is not the case (because a right helicoid is a $\mathcal{C}^1$ surface, but an oblique helicoid is not). Therefore, a right helicoid is never a generalized cylinder, and so the case $d$ and $\Delta$ are orthogonal is also excluded.\\
\newline
Since we covered all the possible cases, the proof of Lemma \ref{LEMMECylinGenerSymetHelic} is complete.
\end{proof}

\noindent
The second result states that if a surface presents two different helical symmetries, then such a surface has to be either a plane, a cylinder of revolution or a sphere.

\begin{lemma}
\label{LEMMEDoublSymetHelic}
Let $\mathcal{S}$ be a $\mathcal{C}^1$ connected surface of the Euclidean space $\mathbb{R}^3$ that is not a plane, and that presents two different helical symmetries (that is, such that either the two axes of the two symmetries are different, or the two shifts of the two symmetries are different, or both the two axes and the two shifts are different).\\
Then, $\mathcal{S}$ is either a cylinder of revolution, or a sphere.
\end{lemma}

\begin{proof}
By assumption, $\mathcal{S}$ presents a first helical symmetry. Up to change the coordinates, there exists a real number $p \in \mathbb{R}$ such that $\mathcal{S}$ presents the helical symmetry of axis $\Delta_1:\text{Span}(e_3)$ and of shift $2\pi p$, where $e_3$ is the third vector of the canonical basis.\\
By assumption, $\mathcal{S}$ presents a second helical symmetry, and so, up to change again the coordinates, there exist three real numbers $\rho$, $\theta$ and $\lambda$ such that the axis of the second helical symmetry is:
\begin{align}
\Delta_2: \begin{pmatrix} \rho \\ 0 \\ 0 \end{pmatrix} + \text{Span} \Big( \begin{pmatrix} 0 \\ \cos\theta \\ \sin\theta \end{pmatrix} \Big)
\end{align}
and the shift of this second symmetry is $2\pi \lambda$.\\
\newline
Therefore, for any point of coordinates $(x,y,z)$ belonging to $\mathcal{S}$, the two vector fields:
\begin{align}
F_1:\begin{pmatrix} x \\ y \\ z \end{pmatrix} \mapsto \begin{pmatrix} 0 \\ 0 \\ 1 \end{pmatrix} \wedge \begin{pmatrix} x \\ y \\ z \end{pmatrix} + \begin{pmatrix} 0 \\ 0 \\ p \end{pmatrix} \hspace{3mm} \text{and} \hspace{3mm} F_2: \begin{pmatrix} x \\ y \\ z \end{pmatrix} \mapsto \begin{pmatrix} 0 \\ \cos\theta \\ \sin\theta \end{pmatrix} \wedge \begin{pmatrix} x - \rho \\ y \\ z \end{pmatrix} + \lambda \begin{pmatrix} 0 \\ \cos\theta \\ \sin\theta \end{pmatrix},
\end{align}
that is:
\begin{align}
F_1(x,y,z) = \begin{pmatrix} -y \\ x \\ p \end{pmatrix} \hspace{5mm} \text{and} \hspace{5mm} F_2(x,y,z) = \begin{pmatrix} \cos\theta z - \sin\theta y \\ \sin\theta(x-\rho) + \lambda \cos\theta \\ -\cos\theta(x-\rho) + \lambda\sin\theta \end{pmatrix}
\end{align}
are tangent to the surface at the point $(x,y,z)$. Therefore, the Lie bracket $[F_1,F_2]$ of $F_1$ and $F_2$ is also tangent to the surface at the point $(x,y,z)$. $[F_1,F_2]$ can be written in coordinates as:
\begin{align}
[F_1,F_2](x,y,z) &= \begin{pmatrix} 0 & -1 & 0 \\ 1 & 0 & 0 \\ 0 & 0 & 0 \end{pmatrix} \begin{pmatrix} \cos\theta z - \sin\theta y \\ \sin\theta(x-\rho) + \lambda\cos\theta \\ -\cos\theta(x-\rho) + \lambda\sin\theta \end{pmatrix} - \begin{pmatrix}0 & -\sin\theta & \cos\theta \\ \sin\theta & 0 & 0 \\ -\cos\theta & 0 & 0 \end{pmatrix}\begin{pmatrix}-y \\ x \\ p \end{pmatrix} \nonumber\\
&= \begin{pmatrix} \sin\theta \rho - \cos\theta (\lambda+p) \\ \cos\theta z \\ -\cos\theta y \end{pmatrix}.
\end{align}
The fact that $[F_1,F_2]$ is tangent to the surface at $(x,y,z)$ can be written as:
\begin{align}
[F_1,F_2](x,y,z) \cdot \left(F_1(x,y,z) \wedge F_2(x,y,z) \right) = 0,
\end{align}
or again, in coordinates:
\begin{align}
\label{EQUATProduScala[F1,F2].(F1^F2)}
&p\cos^2\theta x^2 + \cos\theta \left( \rho\cos\theta + (\lambda-p)\sin\theta\right) yz + p\cos^2\theta z^2 + \left( \rho\sin\theta - (\lambda+p)\cos\theta \right)\left( \rho\cos\theta + (\lambda-p)\sin\theta \right) x \nonumber\\
&\hspace{30mm}+ (-\cos\theta)\left(\rho\sin\theta - \lambda\cos\theta\right)(x^2+y^2) + p \left(\rho\sin\theta - (\lambda+p)\cos\theta \right)\left( \rho\sin\theta - \lambda\cos\theta \right) = 0.
\end{align}
$\mathcal{S}$ being a $\mathcal{C}^1$ surface, it cannot be entirely contained in the axis $\Delta_1$ of the first helical symmetry. Besides, since by assumption $\mathcal{S}$ was assumed not to be a plane, there exists three real numbers $r$, $t_0$ and $z_0$, with $r\neq 0$ and $z_0 \neq 0$ such that the helix described by the following parametric equation:
\begin{align}
\label{EQUATHelicDans___S__}
\left( r\cos(t-t_0),r\sin(t-t_0),z_0+pt\right) \hspace{3mm} \forall t \in \mathbb{R}
\end{align}
is contained in the surface $\mathcal{S}$. Let us observe that the plane $z=0$ corresponds to the plane, orthogonal to the first axis $\Delta_1$, for which the distance between this axis and the other axis $\Delta_2$ is minimal. The information that the surface $\mathcal{S}$ is not entirely contained in this plane will be used in a crucial manner in the following computation.\\
Replacing in the expression \eqref{EQUATProduScala[F1,F2].(F1^F2)} of the scalar product $[F_1,F_2] \cdot \left(F_1 \wedge F_2 \right)$ the variables $x$, $y$ and $z$ by \eqref{EQUATHelicDans___S__}, describing a family of points of the surface $\mathcal{S}$, we obtain:
\begin{align}
\label{EQUATProduScala[F1,F2].(F1^F2)__v2_}
&p\cos^2\theta r^2 \cos^2(t-t_0) + \cos\theta \left( \rho\cos\theta + (\lambda-p)\sin\theta\right) r \sin(t-t_0)(z_0+pt) + p\cos^2\theta (z_0+pt)^2 \nonumber\\
&\hspace{10mm} + \left( \rho\sin\theta - (\lambda+p)\cos\theta \right)\left( \rho\cos\theta + (\lambda-p)\sin\theta \right) r \cos(t-t_0) + (-\cos\theta)\left(\rho\sin\theta - \lambda\cos\theta\right)r^2 \nonumber\\
&\hspace{80mm} + p \left(\rho\sin\theta - (\lambda+p)\cos\theta \right)\left( \rho\sin\theta - \lambda\cos\theta \right) = 0.
\end{align}
\eqref{EQUATProduScala[F1,F2].(F1^F2)__v2_} is an equation of the form:
\begin{align}
c_1 t^2 + c_2 t \sin(t-t_0) + c_3 \cos^2(t-t_0) + c_4 t + c_5 \cos(t-t_0) + c_6 \sin(t-t_0) + c_7,
\end{align}
holding for all $t\in \mathbb{R}$, and where the coefficients $c_i$, with $1 \leq i \leq 7$, are fixed real numbers. Therefore, all these coefficients have to be zero (this can be seen by choosing infinitely many $t-t_0$ such that $\cos(t-t_0) = \pm 1$, so that $\sin(t-t_0) = 0$, and vice versa, providing four polynomial equations in $t$, which enables to conclude). We find therefore the system:
\begin{align}
\label{EQUATSysteCoeffProduScala[F1,F2].(F1^F2)}
\left\{
\begin{array}{rcl}
p^3 \cos^2\theta &=& 0,\\
rp\cos\theta\left(\rho\cos\theta + (\lambda-p)\sin\theta\right) &=& 0,\\
r^2p\cos^2\theta &=& 0,\\
2 z_0 p^2 \cos^2\theta &=& 0,\\
r \left( \rho\sin\theta - (\lambda+p)\cos\theta\right)\left(\rho\cos\theta + (\lambda-p)\sin\theta\right) &=& 0,\\
r z_0 \cos\theta \left(\rho \cos\theta + (\lambda-p)\sin\theta\right) &=& 0,\\
r^2(-\cos\theta)\left(\rho\sin\theta-\lambda\cos\theta\right) + p\left(\rho\sin\theta-(\lambda+p)\cos\theta\right)\left(\rho\sin\theta-\lambda\cos\theta\right) + z_0 p \cos^2\theta &=& 0.
\end{array}
\right.
\end{align}
We consider now the following cases to study the system \eqref{EQUATSysteCoeffProduScala[F1,F2].(F1^F2)}.\\
\newline
First, if $p=0$ (that is, $\mathcal{S}$ is a surface of revolution around the axis $\Delta_1$), \eqref{EQUATSysteCoeffProduScala[F1,F2].(F1^F2)} becomes:
\begin{align}
\left\{
\begin{array}{rcl}
\left( \rho\sin\theta - \lambda\cos\theta\right)\left(\rho\cos\theta + \lambda\sin\theta\right) &=& 0,\\
\cos\theta \left(\rho \cos\theta + \lambda\sin\theta\right) &=& 0,\\
(-\cos\theta)\left(\rho\sin\theta-\lambda\cos\theta\right) &=& 0.
\end{array}
\right.
\end{align}
If in addition $\cos\theta = 0$, we find:
\begin{align}
\rho\lambda \sin^2\theta = \rho\lambda = 0.
\end{align}
If $\rho = 0$, the two axis $\Delta_1$ and $\Delta_2$ are the same, and then $\mathcal{S}$ is a surface of revolution around $\Delta_1$, as well as a surface presenting an helical symmetry around the same axis, with a non trivial shift (because we assumed that $\mathcal{S}$ presents two distinct helical symmetries). Therefore $\mathcal{S}$ is a cylinder of revolution.\\
If on the contrary $\rho \neq 0$, then $\lambda = 0$, but then $\mathcal{S}$ would be a surface of revolution, around two distinct axes $\Delta_1$ and $\Delta_2$ that are parallel. In such a case, $\mathcal{S}$ cannot be a $\mathcal{C}^1$ surface, this case leads to a contradiction.\\
Finally, if we still assume $p=0$, but $\cos\theta \neq 0$, \eqref{EQUATSysteCoeffProduScala[F1,F2].(F1^F2)} becomes:
\begin{align}
\left\{
\begin{array}{rcl}
\rho \cos\theta + \lambda \sin\theta &=& 0,\\
\rho\sin\theta - \lambda \cos\theta &=& 0.
\end{array}
\right.
\end{align}
This last system implies $\rho = \lambda = 0$. In this case, $\mathcal{S}$ is a surface a revolution around two different axes that intersect each other. $\mathcal{S}$ is therefore a sphere.\\
\newline
Second, if $p \neq 0$, the first line of \eqref{EQUATSysteCoeffProduScala[F1,F2].(F1^F2)} implies that
\begin{align}
\cos\theta = 0,
\end{align}
and so the system becomes:
\begin{align}
\left\{
\begin{array}{rcl}
\cos\theta &=& 0,\\
\rho\sin\theta (\lambda-p)\sin\theta &=& 0,\\
\rho^2\sin^2\theta &=& 0.
\end{array}
\right.
\end{align}
We deduce in particular that $\rho = 0$. In this case, $\mathcal{S}$ is a surface presenting two different helical symmetries, around the same axis. Therefore, the shifts of these two helical symmetries are different, and so $\mathcal{S}$ is a cylinder of revolution.\\
\newline
Since we studied all the possible cases, the proof of Lemma \ref{LEMMEDoublSymetHelic} is complete.
\end{proof}

\noindent
We are now in position to state the main result of this section.

\begin{theor}[Characterization of the local Maxwellians solving the free transport equation in a domain with specular reflection boundary condition (SRBC), case $d=3$]
\label{THEORChptrLgTBhSpecuRefleBCd=3}
Let the dimension $d$ be equal to $3$ Let $T$ be a strictly positive number (possibly $+\infty$), let $\Omega$ be a regular, connected open set of $\mathbb{R}^3$ and let $m$ be a local Maxwellian  that is, a function of the form \eqref{EQUATSS2.1ExpreMaxweLocGn}, which solves the transport equation \eqref{EQUATSect2Transport_Libre} on $[0,T]\times\Omega\times\mathbb{R}^2$, with the specular reflection boundary condition (SRBC), that is, such that
\begin{align}
\forall t \in [0,T], x \in \partial \Omega, v \in \mathbb{R}^d,\hspace{3mm} m(t,x,v) = m(t,x,v'),
\end{align}
where
\begin{align}
v' = v - 2\left( v\cdot n(x) \right) n(x)
\end{align}
and with $n(x)$ a unitary normal vector to $\partial\Omega$ at $x$.\\
Then,
\begin{itemize}
\item if $\Omega$ is a half-space, with a boundary of the form $\partial\Omega = \{x_0 n + \lambda \tau\ /\ \lambda \in \mathbb{R}, \tau \in \mathbb{S}^2, \tau\cdot n = 0\}$, with $x_0 \in \mathbb{R}$, $n\in \mathbb{S}^2$ (that is, $x_0 n$ is the closest point of boundary $\partial\Omega$ to the origin, and $\partial\Omega$ is normal to the unit vector $n$), there exist seven real numbers $r_0>0$, $\alpha>0$, $\beta$, $\gamma>0$, $a$, $\ell_1$ ,$\ell_2 \in \mathbb{R}$ and two unit vectors $\tau_1,\tau_2 \in \mathbb{S}^2$ orthogonal to $n$ such that:
\begin{align}
\label{EQUATTheorMaxweSRBC__d=3_1Plan}
\hspace{-5mm} m(t,x,v) &= r_0 \exp\big( - \alpha (x-tv-2x_0 n)\cdot(x-tv) + \beta (x-tv-x_0n)\cdot v - \gamma \vert v \vert^2 +2\left(a n \wedge x \right)\cdot v\nonumber\\
&\hspace{100mm} -2 \ell_1\tau_1\cdot(x-tv) + 2 \ell_2\tau_2\cdot v\big) \hspace{5mm} \nonumber\\
&\hspace{90mm}\forall t \in [0,T], x \in \partial \Omega, v \in \mathbb{R}^d,
\end{align}
\item if $\Omega$ is the volume contained between two parallel planes, with a boundary of the form $\partial\Omega = \{x_1 n + \lambda \tau_1\ /\ \lambda \in \mathbb{R}, \tau \in \mathbb{S}^2, \tau_1\cdot n = 0\} \cup \{ x_2n + \mu \tau_2\ /\ \mu \in \mathbb{R}, \tau_2\in\mathbb{S}^2, \tau_2\cdot n = 0\}$, with $x_0 \in \mathbb{R}$, $n \in \mathbb{S}^1$, there exist five real numbers $r_0 > 0$, $\gamma>0$, $a$, $\ell_1$, $\ell_2 \in \mathbb{R}$ and two unit vectors $\tau_1,\tau_2\in\mathbb{S}^2$ orthogonal to $n$ such that:
\begin{align}
\label{EQUATTheorMaxweSRBC__d=3_2Plan}
\hspace{-6mm} m(t,x,v) &= r_0 \exp\left( - \gamma \vert v \vert^2 + 2\left(an \wedge x\right)\cdot v - 2 \ell_1\tau_1\cdot(x-tv) + 2\ell_2\tau_2\cdot v \right) \hspace{3mm} \forall t \in [0,T], x \in \partial \Omega, v \in \mathbb{R}^d,
\end{align}
\item if $\Omega$ is the volume contained in a cylinder of revolution of affine axis $y + \text{Span}(n)$, or the complement of such a volume, or the volume contained between two cylinders of revolution with the same affine  axis $y + \text{Span}(n)$, with $y \in \mathbb{R}^3$ and $n \in \mathbb{S}^2$, there exist five real numbers $r_0 > 0$, $\gamma > 0$, $a$, $\ell_1$ and $\ell_2$, and a skew-symmetric matrix $\Lambda_0$ such that:
\begin{align}
\label{EQUATTheorMaxweSRBC__d=3_Cylin}
m(t,x,v) &= r_0 \exp\big( - \gamma \vert v \vert^2 + 2a\left(n \wedge(x-y)\right)\cdot v - 2 \ell_1 n\cdot(x-tv) + 2 \ell_2 n \cdot v \big) \nonumber\\
&\hspace{90mm} \forall t \in [0,T], x \in \partial \Omega, v \in \mathbb{R}^d,
\end{align}
\item if $\Omega$ is a sphere centered on $y \in \mathbb{R}^3$, there exist three real numbers $r_0 > 0$, $\gamma > 0$, $a$, and a unit vector $n \in \mathbb{S}^2$ such that:
\begin{align}
\label{EQUATTheorMaxweSRBC__d=3_Spher}
m(t,x,v) &= r_0 \exp\big( - \gamma \vert v \vert^2 + 2a\left(n \wedge (x-y)\right)\cdot v\big) \hspace{5mm} \forall t \in [0,T], x \in \partial \Omega, v \in \mathbb{R}^d,
\end{align}
\item if $\partial\Omega$ presents an helical symmetry of axis $y + \text{Span}(n)$ and of shift $2\pi p$ with $y \in \mathbb{R}^3$, $n \in \mathbb{S}^2$ and $p \in \mathbb{R}$, but such that $\partial\Omega$ is not a sphere nor a generalized cylinder, there exist three real numbers $r_0>0$, $\gamma>0$, $a \neq 0$ and a vector $w \in \mathbb{R}^3$ orthogonal to $n$ such that:
\begin{align}
\label{EQUATTheorMaxweSRBC__d=3_Helic}
m(t,x,v) &= r_0 \exp\left( -\gamma \vert v \vert^2 + 2 a \left(n\wedge(x-y) + pn\right)\cdot v\right) \hspace{5mm} \forall t \in [0,T], x \in \partial \Omega, v \in \mathbb{R}^d,
\end{align}
\item if $\partial\Omega$ is a generalized cylinder of direction $n$ with $n \in \mathbb{S}^2$, but such that $\partial\Omega$ does not present an helical symmetry, there exist four real numbers $r_0 > 0$, $\gamma > 0$, $\ell_1$, $\ell_2$ such that:
\begin{align}
\label{EQUATTheorMaxweSRBC__d=3_GnCyl}
m(t,x,v) &= r_0 \exp\big( - \gamma \vert v \vert^2  - 2 \ell_1 n\cdot(x-tv) + 2 \ell_2 n \cdot v \big) \hspace{5mm} \forall t \in [0,T], x \in \partial \Omega, v \in \mathbb{R}^d,
\end{align}
\item if finally $\partial\Omega$ is not a generalized cylinder and does not present an helical symmetry, then there exist two strictly positive real numbers $r_0 > 0$ and $\gamma > 0$ such that:
\begin{align}
\label{EQUATTheorMaxweSRBC__d=3_Gener}
m(t,x,v) &= r_0 \exp\big( - \gamma \vert v \vert^2  \big) \hspace{5mm} \forall t \in [0,T], x \in \partial \Omega, v \in \mathbb{R}^d.
\end{align}
\end{itemize}
\end{theor}

\begin{remar}
The case when the boundary $\partial\Omega$ of the domain is a plane is the ``most symmetric'' case. Indeed, the domain is not only invariant by translations parallel to the plane or by rotations around the normal lines to the plane, but it is also invariant by dilatations.\\
As a consequence, we obtain in this case a large choice of local Maxwellians that solve the Boltzmann equation with the specular reflection boundary condition. \eqref{EQUATTheorMaxweSRBC__d=3_1Plan} describes a family of local Maxwellians depending on $9$ real parameters.
\end{remar}

\begin{remar}
In the literature (\cite{Desv990}, \cite{ReVi008}), the result presented in Theorem \ref{THEORChptrLgTBhSpecuRefleBCd=3} is restricted to the case when the domain $\Omega$ is bounded. Therefore, the three first cases ($\partial\Omega$ being the union of one or two planes, or of one or two coaxial cylinders) and the sixth case ($\partial\Omega$ is a generalized cylinder) are by definition not considered. In addition, among all the possible helical symmetries, only the case of the rotational symmetry can be considered if $\Omega$ is bounded. In the case when $\partial\Omega$ is bounded, only three possible cases remain, in agreement with the result presented in \cite{Desv990}.
\end{remar}

\begin{proof}[Proof of Theorem \ref{THEORChptrLgTBhSpecuRefleBCd=3}]
As for the $2$-dimensional case, we start from the expression \eqref{EQUATSS2.1ExpreMaxweLocGn} of a local Maxwellian that solves the free transport equation, and the condition \eqref{EQUATChptrLgTBhCond_u(tx)}, that expresses the specular reflection boundary condition. As for the $2$-dimensional case, we obtain the same system \eqref{EQUATChptrLgTBhSyst_u(tx)} on the coefficients $\alpha$, $w_1$, $\Lambda_0$, $\beta$ and $w_2$ of the local Maxwellian $m$.\\
\newline
We start with the first equation \eqref{EQUATChptrLgTBhSystu(tx)1}. If $\alpha \neq 0$, by Lemma \ref{LEMMEChptrLgTBhSRODEAlpha}, the boundary $\partial\Omega$ contains all the half lines of starting point $-w_1/\alpha$ and through $x_0$, for all $x_0 \in \partial\Omega$. Since $\partial\Omega$ is a $\mathcal{C}^1$ surface, $\partial\Omega$ is a plane through the point $-w_1/\alpha$, and therefore $\Omega$ is a half-space.\\
If on the contrary $\alpha = 0$, and if $w_1 \neq 0$, then $\partial\Omega$ is the union of a family of straight lines, all parallel to $w_1$. In other words, $\partial\Omega$ is a generalized cylinder of direction $w_1$.\\
\newline
We now turn to the second equation \eqref{EQUATChptrLgTBhSystu(tx)2}. In a first time, we assume that $\Lambda_0 \neq 0$, and $\beta \neq 0$. Then, from Lemma \ref{LEMMEChptrLgTBhSRODELambd}, for any $x_0 \in \partial\Omega$, $\partial\Omega$ contains a spiral through $x_0$ and inscribed in the plane $\mathcal{P} = \mathcal{P}(\Lambda_0,w_2)$, where $\mathcal{P}$ is the plane through $-y$ and orthogonal to $z$, where $y$ is the preimage of $w_2$ by $\Lambda_0 + \frac{\beta}{2}I_3$ and $z \in \mathbb{R}^3$ is the unique non-zero vector such that $\Lambda_0 x = z \wedge x$ for all $x \in \mathbb{R}^3$. Therefore the point $x_0$ belongs also to the plane $\mathcal{P}$, so that we deduce $\partial\Omega \subset \mathcal{P}$. Since the only $\mathcal{C}^1$ surface contained in a plane is the plane itself, we deduce that $\partial\Omega$ is the plane through $-y$ and orthogonal to $z$.\\%Since $\partial\Omega$ is closed, the center $x_\infty$ of the spiral belongs also to the boundary $\partial\Omega$. But since $\partial\Omega$ is a $\mathcal{C}^1$ surface, the tangent plane to the boundary at $x_\infty$ is equal to the plane in which the spiral is inscribed. In addition, the tangent plane to the boundary at $x_\infty$ is contained in the boundary $\partial\Omega$ itself (because if it would not be the case, it would be possible to find a point $x_0$ of the boundary $\partial\Omega$ outside the tangent plane, but arbitrarily close to it, which would lead to a contradiction because the spiral through $x_0$ would also belong to $\partial\Omega$, so that in a neighbourhood of $x_\infty$, the boundary $\partial\Omega$ would contain two surfaces arbitrarily close to each other, which would contradict the regularity of the boundary $\partial\Omega)$. As a consequence, $\partial\Omega$ is the union of one or two planes orthogonal to $z$.\\
If now $\Lambda_0 \neq 0$ and $\beta = 0$, Lemma \ref{LEMMEChptrLgTBhSRODELambd} provides that $\partial\Omega$ is a surface that presents an helical symmetry, of axis orientated by $z$ (where $\Lambda_0x = z\wedge x \hspace{2mm} \forall x \in\mathbb{R}^3$), and of shift determined by $w_2$.\\
If $\Lambda_0 = 0$ and $\beta \neq 0$, applying Lemma \ref{LEMMEChptrLgTBhSRODEAlpha} we deduce that $\partial\Omega$ is a plane through $-2w_2/\beta$, and therefore that $\Omega$ is a half-space.\\
Finally, if $\Lambda_0 = 0$, $\beta = 0$ and $w_2 \neq 0$, $\partial\Omega$ is a generalized cylinder of direction $w_2$.\\
\newline
Therefore, according to the previous discussion, and the general results of Lemmas \ref{LEMMECylinGenerSymetHelic} and \ref{LEMMEDoublSymetHelic}, we separate the following cases: either $\partial\Omega$ presents no helical symmetry and is not a generalized cylinder, or it is a generalized cylinder without an helical symmetry, or it presents an helical symmetry without being a generalized cylinder, or finally it is a generalized cylinder presenting an helical symmetry, in which case $\partial\Omega$ is either the union of one or two parallel planes, or the union of one or two coaxial cylinders. In addition, when $\partial\Omega$ presents an helical symmetry without being a general cylinder, we will separate the cases when $\partial\Omega$ presents two different helical symmetries or not. In the end, we consider the following cases, that exclude each other, and that cover all the possible situations:
\begin{itemize}
\item $\partial\Omega$ is a single plane,
\item $\partial\Omega$ is the union of two parallel planes,
\item $\partial\Omega$ is the union of one or two coaxial cylinders of revolution,
\item $\partial\Omega$ is a sphere,
\item $\partial\Omega$ is a surface that presents an helical symmetry, but is not a sphere nor a generalized cylinder,
\item $\partial\Omega$ is a generalized cylinder which does not present an helical symmetry,
\item $\partial\Omega$ does not present an helical symmetry and is not a generalized cylinder neither.
\end{itemize}

\noindent
\underline{$\partial\Omega$ is a plane.} Let us assume first that $\partial\Omega$ is a single plane of the form $\{x_0n+\lambda\tau\ /\ \lambda\in\mathbb{R}, \tau\in\mathbb{S}^2, \tau\cdot n = 0\}$, for a certain $x_0 \in \mathbb{R}^3$ that belongs to $\partial\Omega$, and $n \in \mathbb{S}^2$ being normal to $\partial\Omega$.\\
Then, if $\alpha\neq 0$, $-w_1/\alpha \in \partial\Omega$, so that:
\begin{align}
\label{EQUATExpre_w_1_1Plan}
w_1 = - \alpha x_0 n + \ell_1\tau_1,
\end{align}
for some $\ell_1 \in \mathbb{R}$ and $\tau_1 \in \mathbb{S}^2$ with $\tau_1\cdot n = 0$. If on the contrary $\alpha = 0$ and $w_1 \neq 0$, then $w_1$ has to be orthogonal to $n$, hence the general expression \eqref{EQUATExpre_w_1_1Plan} for $w_1$.\\
If $\Lambda_0\neq 0$ and $\beta \neq 0$, we write $\Lambda_0 x = z\wedge x$ for all $x\in\mathbb{R}^3$ with $z\in\mathbb{R}^3$. According to Lemma \ref{LEMMEChptrLgTBhSRODELambd}, $\partial\Omega$ contains a spiral inscribed in a plane $\mathcal{P}$ normal to $z$, and such that $-y$ belongs to the spiral with $\left[ \Lambda_0 + \frac{\beta}{2}I_3\right]y = w_2$. Therefore, we have $z = a n$ for some $a \in \mathbb{R}$, $a\neq 0$, and there exist $\lambda \in \mathbb{R}$ and $\tau_y \in \mathbb{S}^2$ with $\tau_y\cdot n = 0$ such that:
\begin{align}
y = -x_0n + \lambda \tau_y.
\end{align}
As a consequence, we have necessarily:
\begin{align}
w_2 = a\lambda n\wedge \tau_y - \frac{\beta}{2}x_0n + \frac{\beta}{2}\lambda\tau.
\end{align}
Since the matrix $\Lambda_0 + \frac{\beta}{2}I_3$ restricted to $n^\perp$ is invertible, $w_2$ can be written in general as:
\begin{align}
w_2 = -\frac{\beta}{2}x_0 n + \ell_2\tau_2,
\end{align}
for some $\ell_2 \in \mathbb{R}$ and $\tau_2 \in \mathbb{S}^2$ such that $\tau_2\cdot n = 0$.\\
If now $\Lambda_0 \neq 0$ (represented as a vector $z \in \mathbb{R}^3$, $z\neq0$) and $\beta=0$, then according to Lemma \ref{LEMMEChptrLgTBhSRODELambd}, $\partial\Omega$ contains an helix of axis $-y + \text{Span}(z)$, and of shift $2\pi \frac{\lambda}{\vert z \vert}$, with
\begin{align}
w_2 = \lambda \frac{z}{\vert z \vert} + \Lambda_0 y.
\end{align}
Since $\partial\Omega$ is a plane normal to $n$, then necessarily $\lambda = 0$, and there exist $a\in\mathbb{R}$, $a\neq 0$, $\ell_2 \in \mathbb{R}$ and $\tau_2\in\mathbb{S}^2$, $\tau_2\cdot n = 0$, such that:
\begin{align}
z = an,\hspace{5mm} \text{and} \hspace{5mm} w_2 = -\Lambda_0 y = \ell_2\tau_2.
\end{align}
If $\Lambda_0 = 0$ and $\beta \neq 0$, then, according to Lemma \ref{LEMMEChptrLgTBhSRODEAlpha}, $-2\frac{w_2}{\beta}$ belongs to $\partial\Omega$, that is, there exist $\ell_2 \in \mathbb{R}$ and $\tau_2\in\mathbb{S}^2$, $\tau_2\cdot n = 0$, such that:
\begin{align}
w_2 = - \frac{\beta}{2}x_0n + \ell_2\tau_2.
\end{align}
Finally, if $\Lambda_0 = 0$, $\beta = 0$ and $w_2 \neq 0$, then $\partial\Omega$ is a generalized cylinder of direction $w_2$, and so $w_2$ is necessarily orthogonal to $n$.\\
We conclude therefore that if $\partial\Omega$ is a single plane, only the local Maxwellians of the form \eqref{EQUATTheorMaxweSRBC__d=3_1Plan} solve the Boltzmann equation with specular reflection boundary condition.\\
\newline
\underline{$\partial\Omega$ is the union of two parallel planes.} In this case, there exist by assumption one unit vector $n\in\mathbb{S}^2$ and two different real numbers $x_1\neq x_2$ such that:
\begin{align}
\partial\Omega = \{x_1n + \ell_1\tau_1\ /\ \ell_1\in\mathbb{R}, \tau_1\in\mathbb{S}^2, \tau_1\cdot n = 0\} \cup \{x_2n + \ell_2\tau_2\ /\ \ell_2\in\mathbb{R}, \tau_2\in\mathbb{S}^2, \tau_2\cdot n = 0\}.
\end{align}
In this case, according to Lemma \ref{LEMMEChptrLgTBhSRODEAlpha}, $\alpha = 0$ (because $\partial\Omega$ is not composed of a single plane, so $\alpha \neq 0$ leads to a contradiction). If $w_1 \neq 0$, then $w_1$ is necessarily orthogonal to $n$.\\
If $\Lambda_0 \neq 0$ and $\beta \neq 0$, the same  conclusion holds as in the case when $\partial\Omega$ is a single plane. In particular, we obtain a contradiction, since the component of $w_2$ along the direction $n$ should be $-\frac{\beta}{2}x_1$ and $-\frac{\beta}{2}x_2$, but $x_1 \neq x_2$.\\
Similarly, if $\Lambda_0 = 0$ and $\beta \neq 0$, $\partial\Omega$ should be a single plane through $-2\frac{w_2}{\beta}$, which is not the case.\\
We deduce that $\beta = 0$.\\
If $\Lambda_0 \neq 0$, then, according to Lemma \ref{LEMMEChptrLgTBhSRODELambd}, $\partial\Omega$ presents an helical symmetry of axis orientated by $z$ (where $z$ represents $\Lambda_0$). We deduce therefore that $z = a n$, for $a \in \mathbb{R}$, $a\neq 0$. In addition, $\partial\Omega$ being the union of two plans, the shift of this helical symmetry has to be zero, therefore $w_2 = \ell_2\tau_2$, for some $\ell_2 \in \mathbb{R}$, $\tau_2 \in \mathbb{S}^2$ and $\tau_2\cdot n = 0$.\\
If finally $\Lambda_0 = 0$ and $\beta = 0$, $\partial\Omega$ is a generalized cylinder of direction $w_2$, therefore $w_2$ has to be orthogonal to $n$.\\
In summary, we deduce that if $\partial\Omega$ is the union of two different, parallel planes, then only the local Maxwellians of the form \eqref{EQUATTheorMaxweSRBC__d=3_2Plan} solve the Boltzmann equation with the boundary condition \eqref{EQUATChptrLgTBhSpecuRefle_B_C_}.\\
\newline
\underline{$\partial\Omega$ is the union of one or two coaxial cylinders of revolution.} Let us assume that the two cylinders of revolution have the common axis of rotation $y + \text{Span}(n)$, for $y \in\mathbb{R}^3$ and $n \in \mathbb{S}^2$.\\
The following cases lead to a direct contradiction:
\begin{align}
\alpha \neq 0,\hspace{20mm} [\hspace{1mm} \Lambda_0 \neq 0 \hspace{1.5mm} \text{and} \hspace{1.5mm} \beta \neq 0\hspace{1mm}],\hspace{20mm} [\hspace{1mm}\Lambda_0 = 0 \hspace{1.5mm}\text{and}\hspace{1.5mm} \beta \neq 0\hspace{1mm}].
\end{align}
In particular, we deduce that $\alpha = \beta = 0$.\\
In the case when $w_1 \neq 0$, we deduce that $w_1$ has to be colinear to $n$. If $\Lambda_0 \neq 0$, represented by the vector $z$, $\partial\Omega$ presents an helical symmetry of axis $-y' + \text{Span}(z)$, where $w_2 = -\Lambda_0 y' + pn$. Therefore, this implies $z = an$ for a certain $a\in\mathbb{R}$, $a\neq 0$, and $\Lambda_0 y' = \Lambda_0 y$, so that:
\begin{align}
w_2 = -\Lambda_0 y + pn.
\end{align}
If finally $\Lambda = 0$ and $w_2 \neq 0$, then $w_2$ is colinear to $n$. In general, $\Lambda_0$ and $w_2$ have then the forms:
\begin{align}
\Lambda_0x =  an\wedge x \hspace{3mm}\forall x \in \mathbb{R}^3, \hspace{3mm} \text{and} \hspace{3mm} w_2 = -an\wedge y + \ell_2 n,
\end{align}
for $a,\ell_2 \in \mathbb{R}$.\\
Therefore, if $\partial\Omega$ is the union of one or two coaxial cylinders of revolution, any local Maxwellian solving the Boltzmann equation with the boundary condition \eqref{EQUATChptrLgTBhSpecuRefle_B_C_} is necessarily of the form \eqref{EQUATTheorMaxweSRBC__d=3_Cylin}.\\
\newline
\underline{$\partial\Omega$ is a sphere.} Let us assume that the sphere $\partial\Omega$ is centered on $y \in \mathbb{R}^3$. In this case, the following assumptions lead to a direct contradiction:
\begin{align}
&\alpha \neq 0,\hspace{15mm} [\hspace{1mm} \alpha = 0 \hspace{1.5mm}\text{and} \hspace{1.5mm} w_1 \neq 0 \hspace{1mm}],\hspace{15mm} [\hspace{1mm} \Lambda_0 \neq 0 \hspace{1.5mm}\text{and} \hspace{1.5mm} \beta \neq 0 \hspace{1mm}],\nonumber\\
&\hspace{12mm} [\hspace{1mm} \Lambda_0 = 0 \hspace{1.5mm}\text{and} \hspace{1.5mm} \beta \neq 0 \hspace{1mm}],\hspace{15mm} [\hspace{1mm} \Lambda_0 = 0, \hspace{1.5mm} \beta = 0, \hspace{1.5mm}\text{and} \hspace{1.5mm} w_2 \neq 0 \hspace{1mm}].
\end{align}
We deduce in particular that $\alpha = 0$, $\beta = 0$, $w_1 = 0$, and if $\Lambda_0 = 0$, then $w_2 = 0$.\\
In addition, if $\Lambda_0 \neq 0$, represented by the vector $z \in \mathbb{R}^3$, according to Lemma \ref{LEMMEChptrLgTBhSRODELambd}, $\partial\Omega$ presents an helical symmetry, of axis $y + \text{Span}(z)$, with $w_2 = -\Lambda_0 y + pn$. In this case, since $\partial\Omega$ is a sphere, the shift of the helical symmetry has to be zero, so that:
\begin{align}
w_2 = -\Lambda_0 y.
\end{align}
Therefore, if $\partial\Omega$ is a sphere, only the local Maxwellians of the form \eqref{EQUATTheorMaxweSRBC__d=3_Spher} can solve the Boltzmann equation with the boundary condition \eqref{EQUATChptrLgTBhSpecuRefle_B_C_}.\\
\newline
\underline{$\partial\Omega$ presents an helical symmetry, but is not a sphere nor a generalized cylinder.} Let us denote by $y + \text{Span}(n)$ the axis of the helical symmetry of $\partial\Omega$, with $y \in \mathbb{R}^3$ and $n\in\mathbb{S}^2$, and let us denote by $2\pi p$ the shift of this helical symmetry, with $p \in \mathbb{R}$. As for the previous case, the following cases lead to a direct contradiction:
\begin{align}
&\alpha \neq 0,\hspace{15mm} [\hspace{1mm} \alpha = 0 \hspace{1.5mm}\text{and} \hspace{1.5mm} w_1 \neq 0 \hspace{1mm}],\hspace{15mm} [\hspace{1mm} \Lambda_0 \neq 0 \hspace{1.5mm}\text{and} \hspace{1.5mm} \beta \neq 0 \hspace{1mm}],\nonumber\\
&\hspace{12mm} [\hspace{1mm} \Lambda_0 = 0 \hspace{1.5mm}\text{and} \hspace{1.5mm} \beta \neq 0 \hspace{1mm}],\hspace{15mm} [\hspace{1mm} \Lambda_0 = 0, \hspace{1.5mm} \beta = 0, \hspace{1.5mm}\text{and} \hspace{1.5mm} w_2 \neq 0 \hspace{1mm}].
\end{align}
We deduce then that $\alpha = 0$, $\beta = 0$, $w_1 = 0$, and if $\Lambda_0 = 0$ and $\beta = 0$, then $w_2 = 0$ as well.\\
According to Lemma \ref{LEMMEDoublSymetHelic}, since $\partial\Omega$ is not a plane (which is a particular case of a generalized cylinder), nor a cylinder of revolution (another particular case of a generalized cylinder), nor a sphere, we deduce that $\partial\Omega$ presents a single helical symmetry. Therefore, the direction of $z$, representing $\Lambda_0$ is fixed by $n$, and the component of $w_2$ along $n$ is fixed as well by the shift of the helical symmetry. More precisely,  if $\Lambda_0 \neq $, there exists $a\in\mathbb{R}$, $a\neq 0$, such that:
\begin{align}
\Lambda_0 c = a n \wedge x \hspace{2mm} \forall x \in \mathbb{R}^3,
\end{align}
and
\begin{align}
w_2 = -\Lambda_0y + apn = a \left( -n\wedge y + pn\right).
\end{align}
As a consequence, if $\partial\Omega$ presents an helical symmetry, but is not a sphere nor a generalized cylinder, then only the local Maxwellians of the form \eqref{EQUATTheorMaxweSRBC__d=3_GnCyl} solve the Boltzmann equation with the specular reflection boundary condition.\\
\newline
\underline{$\partial\Omega$ is a generalized cylinder which does not present an helical symmetry.} Let us assume that the direction of boundary $\partial\Omega$, as a generalized cylinder, is given by a vector $n\in\mathbb{S}^2$. In the present case, the following cases lead to a direct contradiction:
\begin{align}
&\alpha \neq 0,\hspace{12mm} [\hspace{1mm} \Lambda_0 \neq 0 \hspace{1.5mm}\text{and} \hspace{1.5mm} \beta \neq 0 \hspace{1mm}],\hspace{12mm} [\hspace{1mm} \Lambda_0 \neq 0 \hspace{1.5mm}\text{and} \hspace{1.5mm} \beta = 0 \hspace{1mm}],\hspace{12mm} [\hspace{1mm} \Lambda_0 = 0 \hspace{1.5mm}\text{and} \hspace{1.5mm} \beta \neq 0 \hspace{1mm}].
\end{align}
In particular, we deduce that $\alpha = 0$, $\beta = 0$ and $\Lambda_0 = 0$. In addition, if $w_1$ or $w_2$ are not zero, then they have to be colinear to $n$. Hence the conclusion: only the local Maxwellians of the form \eqref{EQUATTheorMaxweSRBC__d=3_GnCyl} can be solutions of the Boltzmann equation with the boundary condition \eqref{EQUATChptrLgTBhSpecuRefle_B_C_}.\\
\newline
\underline{$\partial\Omega$ does not present an helical symmetry, and is not a generalized cylinder.} In this final case, if any of the parameters $\alpha$, $\beta$, $\Lambda_0$, $w_1$ or $w_2$ is not zero, we obtain a contradiction. Therefore, only the parameter $\gamma$ can be non-zero if a local Maxwellian $m$ solves the Boltzmann equation with the specular reflection boundary condition in such a domain, hence \eqref{EQUATTheorMaxweSRBC__d=3_Gener}.\\
\newline
All the cases were investigated, concluding the proof of Theorem \ref{THEORChptrLgTBhSpecuRefleBCd=3}.
\end{proof}

\noindent
\textbf{Acknowledgements.} The author is grateful to L. Desvillettes, I. Karabash and B. Kepka for fruitful discussions and comments. The author is particularly thankful to I. Karabash for suggesting the use of the Lie brackets to study the surfaces with two helical symmetries.\\
This article originated from a lecture given at the University of Bonn during the Fall semester of 2023. The author is grateful to the students that attended this lecture, in particular to E. Dematt\`e and E. H\"ubner-Rosenau. Their comments helped improving the presentation of some results contained in the present article.\\
The author gratefully acknowledges the financial support, on the one hand when he was affiliated to the University of Bonn until August 2024, of the Hausdorff Research Institute for Mathematics through the collaborative research center The mathematics of emerging effects (CRC 1060, Project-ID 211504053), and the Deutsche Forschungsgemeinschaft (DFG, German Research Foundation), and on the other hand since September 2024, of the project PRIN 2022 (Research Projects of National Relevance) - Project code 202277WX43, based at the University of L’Aquila.

E-mail address: \texttt{theophile.dolmaire@univaq.it}.

\end{document}